\documentclass[11pt,twoside,a4paper]{amsart}
\usepackage{amssymb}
\usepackage{eufrak} 
\date{\today}
\usepackage{geometry}
\def\sta{{\diamond}}

\def\deg{\text{deg}\,}

\def\w{\wedge}

\def\dbar{\bar\partial}

\def\R{{\mathbb R}}
\def\N{{\mathbb N}}
\def\C{{\mathbb C}}
\def\w{{\wedge}}
\def\P{{\mathbb P}}

\def\A{{\mathcal A}}

\def\supp{\text{supp}}

\def\S{{\mathcal S}}

\def\codim{{\rm codim\,}}

\def\E{{\mathcal E}}

\def\O{{\mathcal O}}

\def\Re{{\rm Re\,  }}

\def\U{{\mathcal U}}

\def\J{{\mathcal J}}
\def\nbh{neighborhood }

\def\be{\begin{equation}}
\def\ee{\end{equation}}

\def\Ok{\mathcal O}
\def\mult{{\rm mult}}
\def\lex{\text{lex}}
\def\1{{\bf 1}}

\def\m{{s}}

\newtheorem{thm}{Theorem}[section]
\newtheorem{lma}[thm]{Lemma}
\newtheorem{cor}[thm]{Corollary}
\newtheorem{prop}[thm]{Proposition}

\theoremstyle{definition}

\theoremstyle{remark}

\newtheorem{preremark}[thm]{Remark}
\newtheorem{preex}[thm]{Example}

\newenvironment{remark}{\begin{preremark}}{\qed\end{preremark}}
\newenvironment{ex}{\begin{preex}}{\qed\end{preex}}

\numberwithin{equation}{section}

\title[Segre numbers, a generalized King formula, and local intersections]
{Segre numbers,  a generalized King formula, and local intersections}

\begin{document}

\date{\today}

\author[M.\  Andersson \& H.\  Samuelsson  Kalm \\  \& E.\  Wulcan \& A.\  Yger]
{Mats  Andersson \& H\aa kan  Samuelsson  Kalm\\  \& Elizabeth  Wulcan \& Alain  Yger}

\address{Department of Mathematics\\Chalmers University of Technology and the University of Gothenburg\\S-412 96 
Gothenburg\\SWEDEN}

\email{matsa@chalmers.se \& hasam@chalmers.se \& wulcan@chalmers.se}

\address{Institut de Math\'ematiques, Universit\'e Bordeaux 1, 33405, Talence,
France}
\email{Alain.Yger@math.u-bordeaux1.fr}


\thanks{First three authors partially supported by the Swedish 
  Research Council. Third author also partially supported by the NSF}

\begin{abstract}

Let $\J$ be an ideal sheaf on a reduced analytic space $X$ 
with zero set $Z$.  
We show that the Lelong numbers of the
restrictions to $Z$ 
of certain generalized Monge-Amp\`ere products $(dd^c\log|f|^2)^k$,
where $f$ is a tuple of generators of $\J$, coincide with the
so-called Segre numbers of $\J$, introduced independently by Tworzewski and
Gaffney-Gassler. 
More generally we show that these currents satisfy a generalization of the
classical King formula that  takes into account fixed and moving
components of Vogel cycles associated with
$\J$.
A basic tool is a new calculus for products of positive currents of
Bochner-Martinelli type. 
We also discuss connections to intersection theory.

\end{abstract}

\maketitle

\section{Introduction}\label{intro}

Let $X$  be a reduced analytic space of pure dimension and let $\J$ be a
coherent ideal sheaf on $X$. Given a point $x\in X$, Tworzewski, \cite{T},
Achilles--Manaresi, \cite{AM}, and
Gaffney--Gassler, \cite{GG}, independently 
introduced numbers 
$e_0(\J,X,x),\ldots,e_n(\J,X,x)$
that in a certain sense measure the singularities of $\J$ at $x$ and
that generalize the classical Hilbert-Samuel multiplicity. 
Following \cite{GG} we will
call them {\it Segre numbers}; Tworzewski used the term  
{\it extended index of intersection} whereas Achilles-Manaresi used the term \emph{multiplicity sequence}. 
The definition in \cite{T} goes via a local variant of the
St\"uckrad-Vogel construction, \cite{ST},  introduced in 
 \cite{T,Mass}, and a closely related, also geometric, procedure is used
 in  \cite{GG}. In \cite{AM} the definition is purely algebraic; it is based on Hilbert functions of 
 bigraded rings.
 It is proved in \cite{AR} that the definitions in \cite{T,AM,GG} yield the same numbers; see also \cite{Nowak}.
%
%
%
In this paper we give a (semi-)global representation of these numbers
as the Lelong numbers of certain positive closed currents, constructed
from a tuple of generators of $\J$. This is part of our main result Theorem
\ref{genking}, which is a generalizion of 
King's formula for these currents.

Let us first 
recall the definition 
in \cite{T}. In that paper $X$ is a subvariety of a smooth manifold $Y$
and $\J$ is the pullback to $X$ of the sheaf associated with a smooth submanifold $A\subset Y$. However,
we find it advantageous to avoid any ambient space and instead
consider an arbitrary coherent ideal sheaf $\J\to X$. 
A sequence $h=(h_1,h_2,\ldots, h_n)$ in the local ideal $\J_x$ is called a {\it Vogel sequence of $\J$ at $x$}
 if  there is 
a \nbh $\U\subset X$ of $x$ where the $h_j$ are defined, such that
\begin{equation}\label{vogelcondition}
\codim\big[(\U\setminus Z)\cap(|H_1|\cap\ldots\cap |H_{k}|)\big]=k \ {\rm or} \ \infty, \ k=1,\ldots,n;
\end{equation}
here $Z$ is the (reduced) zero set of $\J$ and $|H_\ell|$ are the supports of the
divisors $H_\ell$ defined by the $h_\ell$. 
Notice that, possibly after shrinking $\U$, the common zero set $\{h=0\}$ in $\U$ equals $Z\cap \U$ and
that if $f_0,\ldots, f_m$ generate $\J_x$, then
any generic sequence of $n$ linear combinations of the $f_j$ is a Vogel sequence 
at $x$.  
Let $X_0=X$ and  let $X_0^Z$ denote the irreducible components of $X_0$ that are contained in $Z$ 
and let $X_0^{X\setminus Z}$ be the  remaining components\footnote{In \cite{GG},  
$X_0^{Z}$ is empty by assumption, but  for us it is convenient
not to exclude the possibility that $\J$ vanishes identically on some irreducible component
of $X$.}, so that 
$$X_0=X_0^Z+  X^{X\setminus Z}_0.$$
By the Vogel condition \eqref{vogelcondition},   $H_1$ intersects $X_0^{X\setminus Z}$ properly. Set  
\begin{equation}\label{strut}
X_{1}=H_1\cdot X^{X\setminus Z}_0
\end{equation}
and decompose analogously $X_1$ into the components $X_1^Z$ contained in $Z$
and the remaining components $X_1^{X\setminus Z}$, so that 
$
X_{1}=X^Z_{1}+X_{1}^{X\setminus Z}. 
$
Define inductively $X_{k+1}=H_{k+1}\cdot X_k^{X\setminus Z}$, $X_{k+1}^Z$, and $X_{k+1}^{X\setminus Z}$. 
Then 
$$
V^h:=X^Z_{0}+ X^Z_{1}+\cdots + X^Z_{n}
$$
is the {\it Vogel cycle}\footnote{
The notion Vogel 
cycle was introduced by Massey \cite{Mass1, Mass}. For a generic choice of Vogel sequence 
the associated Vogel cycle 
coincides with the {\it Segre cycle} introduced by Gaffney and Gassler, \cite{GG}, see Lemma 2.2 in \cite{GG}.}
associated with the Vogel sequence $h$. 
Let $V^h_k$ denote the components of $V^h$ of codimension $k$, i.e.,  $V^h_k=X_k^Z$.  Tworzewski defines the
extended index of intersection 
as 
$\min_\lex (\mult_x V^h_0,\ldots, \mult_x V^h_n)$, 
where the $\min_{\lex}$ is taken over all Vogel sequences $h$ of $\J_x$.

Let us next recall the definition of Segre numbers 
in \cite{GG}, where
also so-called polar multiplicities are introduced. 
Let $f$ be a tuple of generators $f_0,\ldots, f_m$  of $\J_x$ and let $h$
be a Vogel sequence of linear combinations 
$h_j=\alpha_j\cdot f=\alpha_j^0f_0+\cdots +\alpha_j^mf_m$;
notice that any Vogel sequence is on this form for some choice of generators and $\alpha_j$. 
It is proved  in \cite[Section~2]{GG}, see also Section ~\ref{storskurk} below, 
that  the multiplicities $\mult_x V_k^h$ and $\mult_x X_k^{X\setminus Z}$ are
independent of $\alpha_j$ for generic choices of $\alpha_j$ and also
independent of $f$,  
and these numbers are the 
{\it Segre numbers}, $e_k(x)=e_k(\J, X, x)$, and \emph{polar multiplicities},
 $m_k(x)=m_k(\J, X, x)$, respectively. 
%
%
Throughout we will use this definition of the list $e(x)=e(\J,X,x)=\big (e_0(x),\ldots, e_n(x)\big )$ of
Segre numbers. 
In Section \ref{minimal} below we prove that the definitions in
\cite{T} and \cite{GG} coincide, i.e., 
\begin{equation}\label{mintlex}
e(x)=\min_{\lex} \big (\mult_x
V^h_0,\ldots ,\mult_x V^h_n\big ). 
\end{equation}
It is not clear to us 
whether this coincidence 
has been noticed in the literature before. 
In \cite{AR}, both notions are discussed, and \eqref{mintlex} 
follows for the restrictive class of sheaves considered in \cite{T},  but it is not explicitly stated. The coincidence also follows from \cite[Theorem 3.3]{R}
in combination with \cite[Lemma~2.2]{GG}.  

We remark that both definitions above are local. 
Indeed, the Vogel condition
\eqref{vogelcondition} as well as the
genericity of $\alpha_j$ depends on $x$,
cf., Remark~\ref{alvar}. Also the algebraic definition in \cite{AM}
is local.


\smallskip
Let 
$f$ be a tuple of generators of 
the ideal sheaf $\J$. 
For $k=0,1,2,\ldots,n$ we  consider   the closed positive  currents
$M_k^f$ introduced in \cite{A2}. The current $M_k^f$ coincides with
${\bf 1}_Z (dd^c \log |f|^2)^k$, where ${\bf 1}_Z$ is the characteristic function for $Z$ and
\begin{equation}\label{plupp}
(dd^c\log|f|^2)^k:=\lim_{\epsilon\to
  0}\big(dd^c\log(|f|^2+\epsilon)\big)^k; 
\end{equation}
for $k\leq \textrm{codim}\, \{f=0\}$ it is well-known that the definition \eqref{plupp} coincides with the standard one.
Notice that 
$(dd^c \log|f|^2)^0 = 1$ and hence, $M^f_0=\1_Z$
is the current of integration over the components of $X$ that are contained in $Z$; in particular,
it vanishes unless $f\equiv 0$ on some irreducible component of $X$.  
See Section~\ref{bmcurrents} for other expressions for $M^f_k$.  
Our main result is the following. 

\begin{thm}[Generalized King's formula] \label{genking} 
Let $X$ be a reduced  analytic space 
of pure dimension $n$ and let  $\J$ be a coherent ideal sheaf on
$X$ generated by a tuple $f$ of holomorphic functions. 
Let $Z$ be the zero set of $\J$ and let $Z_j^k$ be the distinguished varieties 
of $\J$ of 
codimension $k$.  Then  
\begin{equation}\label{king}
M^f_k={\bf 1}_{Z} (dd^c\log|f|^2)^k=\sum_{j} \beta_j^k [Z_j^k]+N_k^f =: S^f_k+N^f_k,  \quad k=0,\ldots,n,
\end{equation}
where the $\beta_j^k$ are positive integers,  the $N_k^f$ are positive
closed currents, the Lelong numbers $\ell_x(N_k^f)$
are nonnegative integers that only depend on the
integral closure class of $\J$ at $x$, and the set where $\ell_x(N_k^f)\ge 1$ has codimension at least $k+1$. 
 The Lelong number of $M^f_k$ at $x$ is the Segre number $e_k(\J,X,x)$.
The polar multiplicity  $m_k(\J,X,x)$ coincides with the Lelong number at $x$ of the current
$
{\bf 1}_{X\setminus Z}(dd^c\log|f|^2)^k.
$
\end{thm}

Here $[Z_j^k]$ denotes the Lelong current\footnote{We will often identify
a cycle with its corresponding Lelong current.}, i.e., the current of
integration,  associated with the variety $Z_j^k$. 
Recall that the \emph{integral closure} of $\J_x$  
consists of  all holomorphic germs  $\phi$ such that $|\phi|\le C |f|$ for some $C>0$ at $x$.
For the definition of the (Fulton-MacPherson) distinguished varieties $Z_j^k$ at $x$ associated to $\J_x$, 
see Section~\ref{sectiongenking} below. It turns out that 
$S^f=\sum_k S^f_k=\sum_{jk} \beta_j^k Z_j^k$ is precisely the cycle that appears in all Vogel cycles obtained
from  generic (enough) Vogel sequences. It is called the {\it fixed} part in \cite{GG}. 
The remaining parts of the Vogel cycles  vary with the Vogel sequence and are called the 
 {\it moving} parts. Notice that  \eqref{king} is the Siu decomposition, \cite{Siu},   of $M_k^f$.

Note that, contrary to the previous local definitions of Segre numbers, Theorem ~\ref{genking}, gives a
(semi-)global representation of the Segre numbers
\begin{equation}\label{ormens}
e_k(\J,X,x)=\ell_x (M^f_k). 
\end{equation}
Moreover, the $M^f_k$ are obtained as limits of  explicit expressions
in generators of $\J$.





\smallskip
It follows from Lemma~\ref{polkagris} that 
$M_k^f=0$ if $k<\codim Z$ and that $N_{\codim Z}^f=0$. The case 
$k=\codim Z$ of  \eqref{king} is precisely the classical
King formula.



\begin{remark} \label{pia}
If  $\J_x$ is generated by $p<n$ functions $f_0,\ldots, f_{p-1}$, we will see that $M^f_k=0$ for
$k>p$ and hence $e_k=0$ for $k>p$.  However,  
$M^g_k$ may be  non-vanishing if $g$ is another,  larger, set of generators.
If, in addition, $\codim Z_x=p$, i.e., $\J_x$ is a complete intersection, 
then $e_{p}(\J, X, x)$ is the only non-zero entry in $e(\J, X, x)$.
This number  is the classical intersection number 
of the proper intersection 
of the divisors of the $p$ generators $f_j$, see, e.g., \cite{Ch}.
\end{remark}





\begin{cor}\label{smultron} 
If $\J$ is the radical ideal of a variety $Z$ of pure codimension $p$, then 
$M^f_p=[Z]$.
\end{cor}

\begin{remark}\label{alvar}
Assume that   $x$ is  a point where $n_k(\J,X,x)\ge 1$  for some $k$ and let
$V^h$ be a generic Vogel cycle so that $\mult_x V^h_k=e_k(\J,X,x)$. Then
$V^h_k=S^f_k+W$ where the moving part $W$ is a positive cycle of codimension $k$, such that
$\mult_x W=n_k(\J,X,x)$. Since $n_k(\J,X,y)\ge 1$ only on a set
of codimension $\ge k+1$, at most points $y$ on $V^h_k$ we have that
$e_k(\J,X,y)=\mult_y(S^f_k)$ and hence $\mult_y V^h_k > e_k(\J,X,y)$.
As soon as there is a  moving part at $x$ it is thus impossible to find  a Vogel cycle 
that represents the Segre numbers in a whole \nbh of $x$.
\end{remark}

A fundamental ingredient in this paper is a current calculus described in Sections~\ref{lelongcartier}
to \ref{products}. It 
gives an expedient analytic approach to Vogel cycles; for instance, it becomes a straightforward matter
to form mean values of (the Lelong currents of) such cycles. The currents $M^f_k$ are in fact such mean values, 
see Section~\ref{storskurk}; this is the intuitive idea behind
Theorem~\ref{genking}. 
The current calculus is fundamental for the proof of
Theorem~\ref{genking}, which is given in Section~\ref{sectiongenking}, and it
makes it possible to provide a proof of \eqref{mintlex} in our slightly
more general setting of a general sheaf $\J$ than what was considered
in \cite{T}, see Section~\ref{minimal}. 
Our  current calculus  is also useful for concrete computations
of Segre numbers,  see Section~\ref{exsection}.
In Section \ref{invprop} we prove a certain invariance property of Segre
numbers. 
The motivation in  \cite{T} for introducing these numbers was to develop a new intersection theory. 
In Section~\ref{oxet} we discuss some local aspects of connections to intersection theory. 
Our technique to form new currents by averaging Vogel cycles will be the 
starting  point in a forthcoming paper where we will study  a kind of global intersection products.

\begin{remark}
This paper is a shortened and slightly elaborated version of \cite{aswy};
in that paper can be found,  additionally, a discussion of global intersections in the sense of Tworzewski
and various  examples.
\end{remark}

\textbf{Acknowledgement:} We thank Terry Gaffney for fruitful discussions.   
We also thank the referee for important comments on a previous version.
This work was partially carried out while the authors visited the  Mittag-Leffler Institute.

\section{Preliminaries}\label{prelim}
Let us fix some notation. Throughout this paper $X$ is a reduced analytic space 
of pure dimension $n$ and
$\J$ is a coherent ideal sheaf on $X$.
Given a tuple $f=(f_0,\ldots, f_m)$ of holomorphic functions on an analytic space we will use $\J(f)$ to  
denote the sheaf it generates. Similarly if $W\subset X$ is an analytic subset we will use $\J_W$ to denote the 
radical sheaf. 
We will denote the local ring of germs of holomorphic functions at $x$ in $X$ by $\O_{X,x}$.
We say that a sequence $g_1,\ldots, g_m$ of functions on an analytic space $X$ is a \emph{geometrically regular sequence} 
if $\codim\{g_1=\ldots =g_k=0\}=k$ for $1\le k\le m$. 
If $X$ is smooth (or Cohen-Macaulay) a sequence is geometrically regular if and only if it is regular. 

Though less natural at first sight
it is often computationally more convenient to use  regularizations based on analytic continuation
rather than smooth regularizations as in \eqref{plupp}.
For instance, if $h$ is a holomorphic function on $X$ then
\begin{equation*}
\lambda \mapsto \dbar |h|^{2\lambda}\wedge \frac{\partial \log|h|^2}{2\pi i},
\end{equation*}
a priori defined for $\textrm{Re} \, \lambda \gg 0$, has a current-valued analytic continuation to a neighborhood
of $0$ and the value at $0$ is the integration current associated to the divisor defined by $h$.
In general, if $\alpha(\lambda)$ is a current-valued function, defined in a neighborhood of $0$, we let
 $\alpha(\lambda)|_{\lambda=0}$ denote the value at $\lambda=0$.

\subsection{Positive currents}\label{poscurr}
 Let $d^c=(4 \pi i)^{-1}(\partial-\dbar)$ so that 
$dd^c=(2\pi i)^{-1} \dbar \partial$. 
We briefly recall some basic facts about positive currents, referring to  \cite{Ch,Dem}  for details. 
Let $\mu$ be a positive current of bidegree $(k,k)$ defined in some open set $\Omega\subset\C^N$. Then $\mu$ has order zero, 
so that the restriction ${\bf 1}_S\mu$ is well-defined for any Borel set $S\subset\Omega$. If 
in addition $\mu$ is closed and $S$ is analytic,
then the Skoda-El~Mir theorem asserts that  ${\bf 1}_S\mu$ is closed as well.
If  $\mu$ is closed then one can define inductively 
$$
(dd^c\log|z-x|^2)^{j+1}\w\mu=dd^c\big(\log|z-x|^2dd^c((\log|z-x|^2)^{j}\w\mu)\big),
$$
$
(dd^c\log|z-x|^2)^{N-k}\w\mu
$
is a  $(N,N)$-current.  Its  mass at $x$ 
is the \emph{Lelong number} $\ell_x(\mu)$ at $x$ of $\mu$, 
which  depends semi-continuously on  $\mu$, in the sense that
\begin{equation}\label{limsup} 
\ell_x(\mu)\ge\limsup_{j\to\infty}\ell_x(\mu_j)
\end{equation}
if $\mu_j\to \mu$. It  follows that $x\mapsto\ell_x(\mu)$ is upper semi-continuous.

\begin{lma}\label{lelle}
If $\mu$ is a closed positive  $(k,k)$-current in $\Omega\subset\C^N$, then\footnote{By slight abuse of notation we will
write $[0]$, instead of the formally more correct $[\{0\}]$, to denote point evaluation at $0$.}
\begin{equation}\label{system}
\ell_0(\mu)[0]=\lim_{\lambda\to 0^+}\Big(\dbar|z|^{2\lambda}\w\frac{\partial|z|^2}{2\pi i |z|^2}
\w(dd^c\log|z|^2)^{N-k-1}\w\mu\Big).
\end{equation}
\end{lma}

If $k=N$, then the right hand side of \eqref{system} shall be interpreted as 
$$
\lim_{\lambda\to 0^+}\left (1-|z|^{2\lambda}\right )\mu=
{\bf 1}_{\{0\}}\mu,
$$
so  Lemma \ref{lelle} is trivially true  in this case.

\begin{proof}[Sketch of proof]
If $\xi$ is  a test function, then 
\begin{equation}\label{poscurrents1}
\int(dd^c\log|z|^2)^{N-k}\w\mu \wedge\xi=
\lim_{\lambda\to 0^+}\int\frac{|z|^{2\lambda}-1}{\lambda}\w
(dd^c\log|z|^2)^{N-k-1} \wedge \mu \w dd^c \xi.  
\end{equation}
After an  integration by parts, the right-hand side of (\ref{poscurrents1}) may be rewritten as  
\begin{multline*}
\lim_{\lambda\to 0^+}\int\dbar|z|^{2\lambda}\w
\frac{\partial|z|^2}{2  \pi i |z|^2}
\w (dd^c\log|z|^2)^{N-k-1}\w\mu \w \xi \\
  + \lim_{\lambda\to 0^+}\int
|z|^{2\lambda} (dd^c\log|z|^2)^{N-k}\w\mu \w \xi.
\end{multline*}
The second term is precisely the action of ${\bf 1}_{\C^N\setminus\{0\}}(dd^c\log|z|^2)^{N-k}\w\mu$
on $\xi$,  and consequently the point mass at $0$ of 
$(dd^c \log |z|^2)^{N-k}\w \mu$ is the same as the point mass at $0$ of 
the first term, which proves \eqref{system}. 
\end{proof}

\subsection{Currents on an analytic space}\footnote{For a more detailed exposition of
currents on an analytic space we refer to \cite[Section~4.2]{HL}.}\label{curranal}
Let $X$ be a reduced  analytic space of pure dimension $n$. 
Given a local embedding $i\colon X \hookrightarrow \C^N$, we let
$\E_X$ be  the sheaf of smooth forms on $X$, obtained from the sheaf of smooth forms
in the ambient space, where two forms are identified if   their pullbacks
to $X_{reg}$ coincide; it is well-known that this definition 
does not depend on the particular embedding. 
We say that $\mu$ is a \emph{current on $X$ of bidegree $(p,q)$} if
it acts on test forms on $X$ of bidegree $(n-p,n-q)$.
Such currents $\mu$ are naturally identified with currents
$\tau=i_*\mu$  of bidegree
$(N-n+p,N-n+q)$ in the ambient space such that $\tau$ vanish on the kernel of $i^*$.
Observe that the $d$-operator is well-defined on currents on $X$. 
If $W$ is a subvariety of $X$ of pure codimension $p\ge 0$, then
$$
\phi\mapsto [W].\phi=\int_{W_{\rm reg}}\phi
$$ 
is a closed $(p,p)$-current  on $X$; this is the \emph{current of integration} over $W$. 

Recall that a current $\nu$ is \emph{normal} if
both $\nu$ and $d\nu$ have order zero.
The following lemma follows immediately from the corresponding
one in $\C^N$.

\begin{lma}\label{polkagris}
Suppose that $\mu$ is a normal  current of bidegree $(p,p)$ on $X$ that has
support on a  subvariety $W$ of codimension $k$.  If $k>p$ then
$\mu=0$. If $k=p$ and $\mu$ is closed,  
then $\mu=\sum_j \alpha_j[W_j]$ for some  numbers $\alpha_j$, 
where $W_j$ are the irreducible components
of $W$ of codimension $p$.
\end{lma}

It is readily checked that if we have a proper holomorphic mapping $\nu\colon X'\to X$
between analytic spaces, then the push-forward $\nu_*$ is well-defined on
currents on $X'$.

Assume that  $\mu$ is a positive closed current on the analytic space $X$.
Fix $x\in X$ and let $i\colon X\hookrightarrow \C^N$ be a local embedding. We define
the Lelong number $\ell_x(\mu)$ as
$\ell_x(i_*\mu)$. After a suitable change of coordinates $i$ can be factorized as $i=j\circ i'$, where
$i'\colon X\to \C^M$ is a minimal embedding and $j$ is the natural embedding $\C^M\to \C^M\times \C^{N-M}$.
Since the Lelong number  is invariant under  holomorphic changes of coordinates, 
all minimal embeddings are equal  up to a holomorphic change of variables, and
$\ell_x (\tau)=\ell_x(j_* \tau)$,  it follows that $\ell_x(\mu)$ is well-defined.
Thus if  $Z$ is a subvariety of an analytic space $X$ and we have an embedding $X\hookrightarrow \C^N$,
then  the number $\ell_x([Z])$ is independent of whether we 
consider $[Z]$ as the Lelong current of $Z$ on $X$ or on $\C^N$.

Recall  that if $Z$ is a variety  in $\C^N$, 
then the \emph{multiplicity} $\mult_x Z$ of $Z$ at $x$ coincides with the Lelong number
$\ell_x ([Z])$, see \cite[Prop.~3.15.1.2]{Ch}; here $\mult_xZ$ is 
defined as in \cite[Ch.~2.11.1]{Ch}. 
In particular, the Lelong number of the function $1$, considered as a 
current on an analytic space $X$, at $x$ is precisely $\mult_x X$.

The classical \emph{Siu decomposition}, \cite{Siu}, of positive closed currents extends immediately to currents on our 
analytic space $X$. Let $\mu$ be a positive closed $(p,p)$-current on $X$; then there is a unique
decomposition
$$
\mu=\sum_i \beta_i [W_i] +N,
$$
where $W_i$ are irreducible analytic varieties of codimension $p$, $\beta_i\geq 0$, and, for each $\delta>0$, the 
set where $\ell_x(N)\ge \delta$ is analytic and has codimension strictly larger than~$p$.

\subsection{Cycles and Lelong currents}\label{lelongcurrents}
Given an analytic cycle $Z=\sum \alpha_j W_j$, where $W_j$ are varieties,
we let $[Z]=\sum\alpha_j[W_j]$ be  the associated Lelong current.
We will often identify analytic cycles with their Lelong currents.  We let $|Z|$ 
denote the \emph{support} of $Z$,  and we let 
 ${\bf 1}_Z$ mean  ${\bf 1}_{|Z|}$. 
If $H$ is a Cartier divisor  defined by (a germ of) a holomorphic function $h$, we will (sometimes) use the 
notation $[h]$ for $[H]$ and ${\bf 1}_h$ for ${\bf 1}_{|H|}$. 
Given an analytic cycle $Z=\sum \alpha_i W_i$ of pure dimension, the \emph{multiplicity} of $Z$ at $x$ is 
defined as $\sum \alpha_i \mult_x W_i$ (this definition follows \cite[p.~704]{GG}). It follows that
$$
\mult_x Z=\ell_x([Z]).
$$
If $Z=\sum_{k=0}^n Z_k$, 
where $Z_k$ is an analytic cycle of codimension $k$ we define 
\begin{equation}\label{multdef}
\mult_x Z:=(\mult_x Z_0, \ldots, \mult_x Z_n)
\end{equation}
Throughout this paper all analytic cycles are effective, unless otherwise stated.

\subsection{Proper intersections}\label{propinter}
Let $Y$ be a complex  manifold and let $Z_1,\ldots,Z_r$ be (effective) analytic cycles in $Y$ of
pure codimensions $p_j$, $j=1,\ldots ,r$, that intersect properly, i.e., the 
intersection $V$ of their supports has  codimension $p_1+\cdots +p_r$. 
There is a well-defined cycle,  called the (proper) intersection
of the $Z_j$,
\begin{equation}\label{pontus}
Z_r\cdots Z_1=\sum m_j V_j,
\end{equation}
where $V_j$ are the irreducible components of $V$ 
and $m_j$ are certain positive integers. 
One can obtain these numbers $m_j$ by defining the intersection number $i(x)$, algebraically or geometrically, 
at each fixed point
$x$ of $V$, and prove  that $i(x)$  is generically constant on each
$V_j$, see, e.g.,  \cite{Ch}.  However, by means of currents, 
\eqref{pontus} can be obtained  in a more direct way: 
By an appropriate regularization one can define the wedge product
$[Z_r]\w \cdots \w [Z_1]$, see, e.g., \cite{Ch,Dem},
and this current indeed coincides with  the Lelong current of $Z_r\cdots Z_1$.  
In particular,  if the $Z_j$ are (effective) divisors defined by  holomorphic functions $h_j$,
then the Lelong current of the intersection can be obtained explicitly as 
\begin{equation}\label{pontus1}
[Z_r\cdots Z_1]=\lim_{\epsilon\to 0}\bigwedge dd^c\log(|h_j|^2+\epsilon).
\end{equation}
At each point $x$ there is a well-defined  intersection number
$$
\epsilon(x):=\sum_j m_j\mult_x(V_j);
$$
here $\mult_x(V_j)$ is the multiplicity of the variety $V_j$ at $x$.  The  number $\epsilon(x)$ is
precisely equal to the Lelong number $\ell_x([Z_r\cdots Z_1])$ of the positive closed
current $[Z_r\cdots Z_1]$.


\section{Multiplying  a Lelong current by  a Cartier divisor}\label{lelongcartier}
In this section we will describe how the inductive construction of a Vogel cycle $V^h$ can be expediently 
expressed as 
certain products of Lelong currents. Notice that the map $[W]\mapsto \mathbf{1}_Z [W]$ is linear.
Notice also that if $Z,Z'$ are analytic cycles in $X$, then 
\begin{equation}\label{prod1} 
{\bf 1}_{Z'} [Z]=[Z^{Z'}]; 
\end{equation}
recall that $Z^{Z'}$ denotes the irreducible components of $Z$ that are contained in $Z'$. To see \eqref{prod1} 
we may assume 
that $Z$ is irreducible. If $|Z|$ is contained in $|Z'|$, then  ${\bf 1}_{Z'}[Z]=[Z]$. Otherwise,
$|Z|\cap |Z'|$ has higher codimension than $|Z|$, and thus ${\bf 1}_{Z'}[Z]$ vanishes by Lemma \ref{polkagris}. 
Notice that  $\1_Z$ is  $1$ on the components of $X$ that are 
contained in $Z$ and $0$ otherwise, i.e., it is the Lelong current of $X^Z$.

If  $h$ is a non-vanishing holomorphic function on
(each irreducible component of) the analytic space $Z$,
then $\log|h|^2$ is a well-defined $(0,0)$-current on $Z$.
This is clear if $Z$ is smooth and follows in general, 
 e.g., by  means of a smooth resolution $\widetilde Z\to Z$, cf., the proof below.

\begin{lma}\label{snittlemma}
Let $Z$ be an analytic cycle in $X$, $h$ be a holomorphic function, 
and let $u$ be a nonvanishing smooth function on $X$. Then 
\begin{equation}\label{billy}
\lambda\mapsto \dbar|uh|^{2\lambda}\w \frac{\partial \log|uh|^2}{2\pi i} \w[Z],
\end{equation}
a priori defined when $\Re\lambda$ is large, 
has an analytic continuation to a half-plane $\Re \lambda>-\epsilon$, where  $\epsilon >0$. The value 
at $\lambda=0$ is independent of $u$.

If $h$ does not vanish identically on any irreducible component of 
(the support of) $Z$, then this value is equal to
$
dd^c (\log|h|^2 \, [Z]).
$
\end{lma}
Notice that  $v^\lambda:=\dbar |uh|^{2\lambda}\wedge \partial \log |uh|^2/(2\pi i)$
has continuous coefficients when   $\Re\lambda >1$,
so the product in  \eqref{billy} is then  well-defined.

\begin{proof}
First assume that $Z=X=\C^N$ and $h$ is a monomial
$h=z_1^{a_1}\cdots z_N^{a_N}$. Then \eqref{billy} is equal to
$$
v^\lambda=\dbar|uz_1^{a_1}\cdots z_N^{a_N}|^{2\lambda}\w\frac{1}{2\pi i}
\Big[\sum_1^N a_j\frac{dz_j}{z_j}+\frac{\partial|u|^2}{|u|^2}\Big].
$$
One can check that the desired analytic continuation exists, and that the
value at $\lambda=0$ is the current $\sum_1^N a_j [z_j]=dd^c\log |h|^2$; in particular, it is independent of $u$.

Consider now the general case.
By linearity, we may assume that $Z$ is irreducible.
If $h$ vanishes identically on $Z$ and $\Re\lambda$ is large, then $v^\lambda\wedge [Z]=0$, and 
thus it trivially extends to $\lambda\in\C$. Assume that $h$ does not vanish identically on $Z$.  
Let $i\colon Z\hookrightarrow X$ be an embedding and let $\pi\colon \widetilde Z\to Z$ be a  smooth 
modification  of $Z$  
such that $\pi^* i^*h$ is locally a monomial; such a modification exists due to Hironaka's theorem 
on resolution of singularities.
After a partition of unity we are back to the case above. 
It follows that $\pi^*i^*v^\lambda$ has an analytic continuation to $\Re\lambda>-\epsilon$ for some $\epsilon>0$ and thus
$v^\lambda\w[Z]=i_* \pi_*( \pi^*i^*v^\lambda)$ has the desired analytic continuation. 
The value at $\lambda=0$ is equal to 
$$
i_*\pi_*(dd^c\log| \pi^*i^*h|^2)
$$
which proves  the second statement, since $(\log|h|^2)[Z]=i_*\pi_*(\log|\pi^*i^*h|^2)$. 
\end{proof}

Let $H$ denote the Cartier divisor defined by $h$. We define
$[H]\wedge [Z]$ as the value of \eqref{billy} at $\lambda=0$.
According to the lemma it does not depend on the particular choice of $h$  defining $H$.
It follows from the definition that 
\begin{equation}\label{tundra}
[H]\w ([Z_1]+[Z_2])=[H]\w[Z_1]+[H]\w[Z_2]
\end{equation}
and thus $[Z]\mapsto [H]\w [Z]$ is a linear operator
on Lelong currents, cf., \eqref{prod1}. 
However, in general it is {\it not} true that
$([H_1]+[H_2])\wedge [Z]= [H_1]\w [Z] +[H_2]\w [Z]$ or
$[H_1]\wedge[H_2] = [H_2]\wedge [H_1]$.

Since \eqref{billy} is analytic by Lemma~\ref{snittlemma} it follows that
$[H]\wedge {\bf 1}_H[Z]=0$, and so, using \eqref{tundra}, we get
\begin{equation*}
[H]\wedge[Z]=[H]\wedge \big({\bf 1}_H [Z] + {\bf 1}_{X\setminus H}[Z]\big) = 
[H]\wedge {\bf 1}_{X\setminus H}[Z].
\end{equation*}
Now, ${\bf 1}_{X\setminus H}[Z]$ is the current of integration over $Z^{X\setminus H}$, i.e.,
the components of $Z$ that are not contained in $|H|$. Since $H$ and $Z^{X\setminus H}$
intersect properly it follows that $[H]\wedge {\bf 1}_{X\setminus H}[Z]=[H\cdot Z^{X\setminus H}]$, 
cf.\ Section~\ref{propinter}. Summarizing, we get the computation rules
\begin{equation}\label{rakneregel} 
[H]\w [Z]=[H]\w \1_{X\setminus H}[Z]= [H\cdot Z^{X\setminus H}].
\end{equation}

\begin{remark} It is important to emphasize  that $[H]\w [Z]$ is {\it not} the same as
(the Lelong current associated with)  
the intersection  $H\cdot Z$   in \cite{fult}. In fact, if $Z$ is irreducible and
contained in $H$, then $[H]\w [Z]=0$, whereas in \cite{fult} the product
is a cycle in $Z$ of codimension $1$ that is well-defined up to rational equivalence.
\end{remark}

\begin{ex} \label{prodex}
Let $H_1$ and $H_2$ be Cartier divisors and let $H=H_1+H_2$. Then
$[H_1]\w [H]=[H_1]\w [H_2]$ but $[H]\w [H_1]= 0$. 
Moreover $[H_1]\wedge  {\bf 1}_{H_1} [H]=[H_1]\w [H_1]=0$ but ${\bf 1}_{H_1} [H_1]\w [H]=
{\bf 1}_{H_1} [H_1]\w[H_2]=[H_1]\w[H_2]$.
\end{ex}

We can construct Vogel cycles, cf., Section~\ref{intro},  by inductively applying operators $\1_Z$ and $[H]\w$. 
\begin{prop}\label{produktprop}
Let $X$ be an analytic space of dimension $n$ and let $h=(h_1,\ldots, h_n)$ be a Vogel sequence 
of an ideal $\J$ with variety $Z$ at $x\in X$, with corresponding divisors $H_1,\ldots, H_n$. 
Then on $X$,
\begin{equation}\label{pall}
[X_0]=1,\quad  
[X_{\ell}]=[H_\ell]\w\cdots\w [H_1]\,,\ \ell=1,\ldots, n 
\end{equation}
and  
\begin{equation}\label{pall2}
[X_0^Z]={\bf 1}_Z,   \quad [X^Z_{\ell}]={\bf 1}_Z  [H_\ell]\w\cdots\w [H_1],  \ \ell=1,\ldots, n.
\end{equation}
In particular, 
\begin{equation}\label{leongvogel}
[V^h]={\bf 1}_Z + {\bf 1}_Z[H_1]+ {\bf 1}_Z[H_2]\w[H_1] + \cdots +
{\bf 1}_Z[H_n]\w\cdots\w[H_1]. 
\end{equation}
\end{prop}

If we consider $X$ as embedded in some larger analytic space $Y$, then we have instead
$$
[X_0]= [X],\quad [X_{\ell}]=[H_\ell]\w\cdots\w [H_1]\w [X], \ \ell=1,\ldots, n
$$
and
$$
[X^Z_0]= {\bf 1}_Z [X],\quad [X^Z_{\ell}]={\bf 1}_Z  [H_\ell]\w\cdots\w [H_1]\w [X],\  \ell=1,\ldots, n
$$

\begin{proof}
In view of \eqref{prod1}, \eqref{pall2} follows from \eqref{pall}. 
Using \eqref{tundra}, we have, in view of \eqref{strut}, that
$$
[X_{1}]=[H_1]\w [X^{X\setminus Z}_0]=[H_1]\w([X_0]-[X^Z_{0}])=[H_1]
$$
since $[H_1]\w[X^Z_{0}]=[H_1]{\bf 1}_Z=0$. One obtains \eqref{pall} by induction. 
\end{proof}

\section{Bochner-Martinelli currents}\label{bmcurrents} 
Let $f=(f_0,\ldots, f_m)$ be a tuple of holomorphic functions on $X$ that generates $\J$
and let $Z$ be the zero set of $\J$. 
For $\Re \lambda \gg 0$, let 
\begin{eqnarray*}
&&M^{f,\lambda}_0:=1-|f|^{2\lambda}\\
&&M^{f,\lambda}_k:=\dbar|f|^{2\lambda}\w\frac{\partial\log|f|^2}{2\pi i}\w
(dd^c\log|f|^2)^{k-1} \text{ if } k\ge 1,
\end{eqnarray*} 
and
\begin{equation}\label{summan}
M^{f,\lambda}:=\sum_{k=0}^\infty M^{f,\lambda}_k,
\end{equation}
where $|f|^2=\sum_{j=0}^m|f_j|^2$. 
The sum in \eqref{summan} is finite for degree reasons,
and when   $\Re\lambda\gg0$, $M^{f,\lambda}$ is locally integrable. 
We will show that $\lambda\mapsto M^{f,\lambda}_k$ has a current-valued analytic continuation  to $\Re\lambda>-\epsilon$, 
for some $\epsilon >0$.  
We denote the value of $M^{f,\lambda}_k$ at $\lambda=0$ by $M_k^f$ and we write 
$
M^{f}:=\sum_{k}M_k^{f}.
$
The current $M^f$ and its components $M^f_k$ will be referred to as {\it Bochner-Martinelli currents}, 
cf. Remark \ref{oriental} below. 

A computation yields that 
\begin{equation*}
M^{f,\lambda}_k=
\lambda\frac{i}{2\pi}\frac{\partial |f|^2\w \dbar |f|^2}{|f|^{4-2\lambda}}\w
(dd^c\log |f|^2)^{k-1}
\end{equation*}
which is positive when $\lambda>0$, and thus $M^f_k$ is a positive current. 
Note that $M_0^f$ is the current of integration over the components of $X$, on which $f\equiv 0$. 
In particular, if $f$ does not vanish identically on any component of $X$, then $M_0^f=0$.  

Let  $\pi\colon \widetilde X\to X$ be a normal modification  
such that the pull-back ideal 
sheaf $\J\cdot \mathcal O_{\widetilde X}$ 
is principal; for instance one can take the normalization of the blow-up of
$X$ along $\J$.  
Then  $\pi^*f=f^0 f'$ where $f^0$ is a section of the holomorphic 
line bundle $L\to\widetilde X$ corresponding to the exceptional divisor $D_f$
of $\pi\colon \widetilde X\to X$, i.e., the divisor defined by $\mathcal J\cdot \mathcal O_{\widetilde X}$, and $f'$ 
is a nonvanishing tuple of sections
of $L^{-1}$.  Let $L$ be equipped with the metric defined by 
$|f^0|_{L}=|\pi^* f|=|f^0f'|$, and let 
\begin{equation}\label{omegadef}
\omega_f:=dd^c\log|f'|^2;
\end{equation}
here the right hand side is computed locally for any local trivialization of $L^{-1}$. Then $-\omega_f$ is the first 
Chern form of $(L,|\cdot|_L)$, and clearly
$\omega_f\ge 0$. 

Since $\log |\pi^* f|^2=\log |f^0|^2+\log |f'|^2$ it follows from the Poincare-Lelong formula that 
\begin{equation}\label{potta}
dd^c\log|\pi^* f|^2=[D_f]+\omega_f.
\end{equation}
In particular,
$
\pi^* (dd^c\log|f|^2)=\omega_f
$
outside $\pi^{-1}\{f=0\}$.  
Therefore, for $\Re \lambda \gg0$,
\begin{eqnarray}\label{billdalar}
&&\pi^* M^{f,\lambda}_0=1-|f^0 f'|^{2\lambda}\\\label{billdal}
&&\pi^* M^{f,\lambda}_k=(2\pi i)^{-1}\dbar|f^0f'|^{2\lambda}\w \partial\log|f^0f'|^2\w
\omega_f^{k-1}, ~~~ k\geq 1.
\end{eqnarray}
Now Lemma~\ref{snittlemma} asserts that 
$\lambda\mapsto \pi^* M^{f,\lambda}_k$  has an analytic continuation to $\Re\lambda >-\epsilon$ and 
since $M^{f,\lambda}_k=\pi_*\pi^* M^{f,\lambda}_k$ for $\Re \lambda \gg 0$, it follows that $\lambda\mapsto M^{f,\lambda}_k$ 
has the desired analytic continuation. Moreover 
\begin{eqnarray}
\label{MAlikhet0}
&&M_0^f = M_0^{f,\lambda}|_{\lambda=0}= 
\pi_*(\pi^* M_0^{f,\lambda}|_{\lambda=0}) = \pi_*({\bf 1}_{D_f})={\bf 1}_{\{f=0\}}\,. 
\\ \label{data2}
&&M^f_k =M^{f,\lambda}_k|_{\lambda=0} = \pi_*(\pi^*M^{f,\lambda}_k|_{\lambda=0})=
\pi_*([D_f]\w\omega_f^{k-1}), ~~~k\geq 1.
\end{eqnarray}

Following for example \cite{A2} one can check that for $k\geq 1$,
\begin{equation}\label{MAlikhet}
M_k^f={\bf 1}_Z (dd^c\log|f|^2)^k 
\end{equation}
and
$$
{\bf 1}_{X\setminus Z} (dd^c\log|f|^2)^k=\pi_*(\omega_f^k).
$$

It is not hard to see that  $M^{f,\lambda}_k$ is locally integrable for $\Re\lambda >0$ and that
$M^{f,\lambda}_k\to M^f_k$ as measures when $\lambda\to 0^+$.

\begin{remark}\label{gammaremark}
For future reference, let $g$ be a tuple of holomorphic functions such that $|g|\sim |f|$, i.e., 
there exists $C\in\R$ such 
that $|f|/C\leq |g|\leq C|f|$, 
and let $\pi\colon \widetilde X \to X$ be a normal modification such that both $\J(f)\cdot \O_{\widetilde X}$ 
and $\J(g)\cdot\O_{\widetilde X}$ are principal. Then $|f^0f'|\sim |g^0g'|$ and since $f'$ and $g'$ are 
non-vanishing 
it follows that $f^0$ and $g^0$ define the same divisor on $\widetilde X$. Therefore the corresponding
negative Chern forms $\omega_f$ and $\omega_g$ are $dd^c$-cohomologous, i.e.,  
there is a global smooth function $\gamma$ such that $dd^c\gamma=\omega_f-\omega_g$.  
\end{remark}

By combining \cite[Proposition~3.2]{A2} and \cite[Corollary~4]{hasamJFA} it follows that
\begin{equation}\label{mmm}
M^f_k=\lim_{\epsilon\to 0} \,  \frac{\epsilon(dd^c|f|^2)^k}{(|f|^2+\epsilon)^{k+1}}.
\end{equation}

\begin{remark}\label{oriental} Given a tuple $f$ as above, associated residue currents of Bochner-Martinelli type 
were introduced in \cite{PTY}. Let $E$ be a trivial vector bundle with basis elements $e_0,\ldots, e_m$
and consider $f=f_0e_0+\cdots +f_me_m$ as a section of the dual bundle $E^*$ with basis elements $e_j^*$.
Following \cite{A1} one can define a residue current $R^f=R^f_0+\cdots + R^f_n$, where $R^f_k$
is a current of bidegree $(0,k)$ with values in  the exterior product $\Lambda^k E$, such that the coefficients
in $R^f$ are precisely the currents in \cite{PTY}.  It is proved in \cite{A2} that
$$
M^f_k=R^f\cdot (df)^k/(2\pi i)^kk!,
$$
where $\cdot$ denotes the natural contraction.  For more details, see, e.g., \cite{A2}.
\end{remark}


\section{Products of Bochner-Martinelli currents}\label{products} 
Given tuples $f_1,\ldots, f_r$ of holomorphic 
functions in $X$, we will give meaning to the product 
\begin{equation}\label{klamma}
M^{f_r}\w\cdots\w M^{f_1} 
\end{equation} 
of Bochner-Martinelli currents. The construction is recursive. 
Assume that $M^{f_\ell}\w\cdots\w M^{f_1}$ is defined;
it follows from the proof of Proposition \ref{prop3-1} that
\begin{equation}\label{forts}
\lambda\mapsto 
M^{f_{\ell+1},\lambda}\w M^{f_\ell}\w\cdots \w M^{f_1}
\end{equation}
is holomorphic for $\Re\lambda>-\epsilon$, where $\epsilon >0$. Set   
\begin{equation}\label{defen}
M^{f_{\ell+1}}\w M^{f_\ell}\ldots\w M^{f_1}:=
M^{f_{\ell+1},\lambda}\w M^{f_\ell}\w\ldots \w M^{f_1}\big|_{\lambda=0}.
\end{equation}
We define the products $M^{f_r}_{k_r}\w\cdots\w M^{f_1}_{k_1}$ in the analogous way so that 
\begin{equation}\label{devproduct}
M^{f_r}\w\cdots\w M^{f_1}=
\sum_{k_r,\ldots,k_1\ge 0}
M^{f_r}_{k_r}\w\cdots\w M^{f_1}_{k_1}. 
\end{equation}

Notice that if the $f_j$ are single functions, then $M^{f_j}={\bf 1}_{f_j} + [f_j]$ and 
\begin{equation}\label{krabba}
M^{f_r}\wedge \cdots\wedge M^{f_1}= ({\bf 1}_{f_r} + [f_r])\wedge\cdots\wedge ({\bf 1}_{f_1} + [f_1]); 
\end{equation} 
cf.\ Section~\ref{lelongcartier}.

\begin{prop}\label{vogelex}
If $h=(h_1,\ldots,h_n)$ is a Vogel sequence of some ideal with zero set $Z$ at $x$, then
\begin{equation*}
M^{h_n}\wedge \cdots \wedge M^{h_1} = [V^h].
\end{equation*}
\end{prop}

\begin{proof}
In light of Lemma ~\ref{polkagris}, 
\begin{equation*}
\1_{h_\ell}\cdots \1_{h_{k+1}}[h_k]\w \cdots\w [h_1]=\1_Z [h_k]\w\cdots\w [h_1].
\end{equation*} 
Thus, by \eqref{rakneregel}, 
$[h_{\ell+1}] \w \1_{h_\ell}\cdots \1_{h_{k+1}}[h_k]\w\cdots\w [h_1]=0$. Hence, in view of \eqref{krabba},
\begin{equation}\label{italien}
M^{h_n}\w\cdots\w M^{h_1}= \sum_{k=0}^n {\bf 1}_{h_n}\cdots {\bf 1}_{h_{k+1}} [h_k]\w\cdots \w [h_1] = 
\sum_{k=0}^n {\bf 1}_{Z} [h_k]\w\cdots \w [h_1];
\end{equation}
here we have used that $[h_n]\w\cdots\w [h_1]$ has support on $Z$. 
Now, Proposition~\ref{produktprop} asserts that the right hand side of \eqref{italien} is equal to $[V^h]$.
\end{proof}

In contrast to the definition of products of residue currents of Bochner-Martinelli
type introduced in \cite{Wulcan}, the recursively defined products \eqref{klamma} are \emph{not}
commutative in general, not even if the tuples just consist of one single function. For instance, in $\C^2_{x,y}$
we have that $M^{xy}\w M^{y}=0$, whereas $M^{y}\w M^{xy}=[0]$, cf., Example ~\ref{prodex}.
Various approaches to recursively defined products of residue currents are investigated in \cite{LS}.

\begin{prop}\label{prop3-1} 
Let $f_1,\ldots, f_r$ be tuples of holomorphic functions in $X$, with common zero set $Z=\{f_1=\ldots =f_r=0\}$. 
Then the current $M^{f_r}\w\cdots\w M^{f_1}$, defined by \eqref{defen}, is positive and  has support on $Z$.

Let $\pi \colon \widetilde X\to X$ be a normal modification such that the sheaves $\J(f_\ell)\cdot\O_{\widetilde X}$ 
are principal for $\ell=1,\ldots r$. As in Section \ref{bmcurrents},  let  $D_{f_\ell}$ 
and $\omega_{f_\ell}$ be  the corresponding divisors and negative Chern forms, respectively. Then 
\begin{equation}\label{uppeform}
M^{f_r}_{k_r}\w\ldots\w M^{f_1}_{k_1}
= \pi_*\big([D_{f_r}]\w\cdots\w[D_{f_1}]\w
\omega_{f_r}^{k_r-1}\w\cdots\w\omega_{f_1}^{k_1-1}\big),
\end{equation}
where, if $k_\ell=0$, the factor $[D_{f_\ell}]$ shall be replaced by ${\bf 1}_{D_{f_j}}$ and the factor $\omega_{f_\ell}^{k_\ell-1}$ 
shall be removed.

Assume that $g_1,\ldots g_r$ are tuples of holomorphic functions in $X$ such that $|g_\ell|\sim|f_\ell|$ for $\ell=1,\ldots,r$. 
Then there is a normal current $T$ with
support on $Z$ such that
\begin{equation}\label{prop3.1-1graderad}
dd^c T= M^{f_r}_{k_r}\w\cdots\w M^{f_1}_{k_1}-M^{g_r}_{k_r}\w\cdots\w M^{g_1}_{k_1}\,.  
\end{equation}
\end{prop}

\begin{proof} 
Iteratively using Lemma \ref{snittlemma}, the computation rules \eqref{rakneregel}, 
and \eqref{billdalar}--\eqref{data2} we see that the desired 
analytic continuation of \eqref{forts} exists and that \eqref{uppeform} holds. 
It follows that $M^{f_r}_{k_r}\w\ldots\w M^{f_1}_{k_1}$ has its support 
contained in  $\pi(|D_{f_r}|\cap\cdots\cap |D_{f_1}|)=Z$. 
Moreover $M^{f_r}_{k_r}\w\ldots\w M^{f_1}_{k_1}$ is the push-forward of a product of positive $(1,1)$-currents and 
positive forms, and hence it is positive.

To prove the last  part, it suffices to change
one of the $f_\ell$ to $g_\ell$ with $|g_\ell|\sim|f_\ell|$. First notice that 
then $M^{f_\ell}_0={\bf 1}_{f_\ell}={\bf 1}_{g_\ell}=M^{g_\ell}_0$.
Let us then  assume that  $k_\ell\geq 1$, and that 
the modification $\pi$ is chosen so that also   $\J(g_\ell)\cdot \O_{\widetilde X}$ is principal. 
By Remark~\ref{gammaremark}, there is a smooth global function
$\gamma$ on $\widetilde X$ such that $\omega_{f_\ell}-\omega_{g_\ell}=dd^c\gamma$ and thus we can find a smooth 
global form $w$ such that
$
dd^c w=\omega_{f_\ell}^{k_\ell-1}-\omega_{g_\ell}^{k_\ell-1}. 
$
Let 
$$
T: = \pi_*\big(\tau_r\w\cdots\w \tau_{\ell+1}\w [D_{f_\ell}]\w w \w \tau_{\ell-1}\w\cdots\w \tau_1),
$$
where $\tau_j= {\bf 1}_{D_{f_j}}$ if $k_j=0$ and $\tau_j=[D_{f_j}]\wedge \omega_{f_j}^{k_j-1}$ otherwise. 
Then $T$ satisfies \eqref{prop3.1-1graderad}. Note that 
$\tau_r\w\cdots\w \tau_{\ell+1}\w [D_{f_\ell}]\w w \w \tau_{\ell-1}\w\cdots\w \tau_1$ is normal, 
and since normality is preserved under push-forward, so is $T$. 
\end{proof}

We also define products of Bochner-Martinelli currents and Lelong currents. If $f_1,\ldots, f_r$ are 
tuples of holomorphic functions in $X$ and $Z$ is an analytic subset of $X$, we define recursively 
$M^{f_1}\w [Z]:=M^{f_1,\lambda }\w [Z]\big|_{\lambda=0}$,
and 
$$
M^{f_{k+1}}\w \cdots\w M^{f_1}\w [Z]:=
M^{f_{k+1},\lambda}\w M^{f_k}\w\cdots \w M^{f_1}\w [Z]\big|_{\lambda=0}.
$$
By arguments as in the proof of Proposition \ref{prop3-1} one can prove that the desired analytic continuations exist, 
and thus $M^{f_r}\w\cdots\w M^{f_1}\w [Z]$ is well-defined. 
It is  readily checked that if $i\colon Z\hookrightarrow X$, then, for any $k_1,...,k_r\in \N$, 
\begin{equation}\label{jobb}
M^{f_r}_{k_r}\w\cdots\w M^{f_1}_{k_1}\w [Z]=
i_* [M^{i^*f_r}_{k_r}\w\cdots\w M^{i^*f_1}_{k_r}].
\end{equation}
Moreover, if $Z=Z'+Z''$, $Z''\subset \{f_j=0\}$, and $k_j>0$, then one checks that
\begin{equation}\label{jobb2}
M^{f_r}_{k_r}\w\cdots\w M^{f_1}_{k_1}\w [Z]=
M^{f_r}_{k_r}\w\cdots\w M^{f_1}_{k_1}\w [Z'],
\end{equation}
cf.\ (the first equality of) \eqref{rakneregel}.

For future reference, note that if $f$ is a tuple of holomorphic functions on the analytic space $X$ then 
\begin{equation}\label{uppdelning}
M^f=M^f\1_X=\sum_jM^f\1_{X_j},
\end{equation}
where $X_j$ are the irreducible components of $X$.

\begin{prop}\label{heldag}
Let $f_1,\ldots, f_r$ be tuples of holomorphic functions in $X$ and let $\xi$ be a tuple of holomorphic functions 
such that $\{\xi=0\}=\{x\}$, where $x\in X$. 
Then 
\begin{equation}\label{snurragrad}
M^\xi\w M^{f_r}_{k_r}\w\cdots\w M^{f_1}_{k_1}=M^\xi_{n-k} \w M^{f_r}_{k_r}\w\cdots\w M^{f_1}_{k_1}=\alpha[x],
\end{equation}
where $k=k_1+\cdots + k_r$ and $\alpha$ is a non-negative integer. If $\xi$ generates the maximal ideal at $x\in X$, 
then $\alpha=\ell_x\big(M^{f_r}_{k_r}\w\cdots\w M^{f_1}_{k_1}\big)$. 
\end{prop}

\begin{proof}
By Proposition \ref{prop3-1}, $M^\xi_{n-k}\w M^{f_r}_{k_r}\w\cdots\w M^{f_1}_{k_1}$ is positive and has support at $x$, 
and thus by Lemma \ref{polkagris} it is of the form $\alpha[x]$ for some non-negative $\alpha$. 
Let $\pi:\widetilde X\to X$ be a normal modification such that $\J(f_\ell)\cdot\O_{\widetilde X}$ and 
$\J(\xi)\cdot\O_{\widetilde X}$ are principal. 
Let us use the notation from Section \ref{bmcurrents}. 
Then, from \eqref{uppeform}, we see that $\alpha$ is an intersection number and hence an integer.

Now assume that $\xi$ generates the maximal ideal at $x$ and that $i\colon X\hookrightarrow\C^N$ is a local
embedding such that $i(x)=0$, so that $i_*[x]=[0]$. By the second part of Proposition \ref{prop3-1} we may assume 
that $f_j=i^* F_j$ and $\xi=i^* z$ for some tuples $F_j$ and 
the standard  coordinate system $z=(z_1,\ldots,z_N)$ in $\C^N$. 
%
Then
\begin{equation}\label{uggl}
i_*( M^\xi_{n-k}\w M^{f_r}_{k_r}\w\cdots\w M^{f_1}_{k_1})=
M^z_{n-k}\w M^{F_r}_{k_r}\w\cdots\w M^{F_1}_{k_1}\w [X],
\end{equation}
cf. \eqref{jobb}. 
By Lemma \ref{lelle}, the right hand side of \eqref{uggl} is precisely the Lelong 
number of $M^{F_r}_{k_r}\w\cdots\w M^{F_1}_{k_1}\w [X]$ at $0$ in $\C^N$ times $[0]$. 
\end{proof}

\begin{prop}\label{remarkheldag} 
The  Lelong number at $x$ of $M^{f_r}_{k_r}\w\cdots\w M^{f_1}_{k_1}$
is unchanged if we replace $f_j$ by $g_j$ such that $|f_j|\sim |g_j|$.
\end{prop}

\begin{proof}
It follows from the second part of Proposition~\ref{prop3-1},  
applied to $f_1,\ldots, f_r,\xi$ since then 
$T$, which has  bidegree  $(n-1,n-1)$, must vanish by Lemma~\ref{polkagris}.
\end{proof}

One can replace all the evaluations in the definition of the product by one single
evaluation in the following way; for the proof see \cite{aswylambda}.

\begin{prop}\label{prop3-4} 
Assume that $\mu_j$ are strictly positive integers such that
$\mu_1>\mu_2>\ldots >\mu_r$. Then $\lambda\mapsto M^{f_r,\lambda^{\mu_r}}_{k_r}\w\cdots\w M^{f_1,\lambda^{\mu_1}}_{k_1}$
is holomorphic in a \nbh of the half-axis $[0,\infty)$ in $\C$ and 
\begin{equation}\label{product} 
M^{f_r}_{k_r}\w\cdots\w M^{f_1}_{k_1} = M^{f_r,\lambda^{\mu_r}}_{k_r}\w\cdots\w M^{f_1,\lambda^{\mu_1}}_{k_1}\big|_{\lambda=0}\,.
\end{equation}  
\end{prop}

In view of Proposition~\ref{vogelex} we get

\begin{cor}\label{vogelbike}
If $h_1,\ldots,h_n$ is a Vogel sequence of some ideal at some point $x$
and $\mu_j$ are as in Proposition~\ref{prop3-4},  then
 the Lelong current of the
associated Vogel cycle is given as the value at $\lambda=0$ of the function
$$
\lambda \mapsto \bigwedge_{k=1}^n M^{h_k, \lambda^{\mu_k}}= \bigwedge_{k=1}^n
\big( 1-|h_k|^{2\lambda^{\mu_k}}+\dbar|h_k|^{2\lambda^{\mu_k}}\w
\partial \log|h_k|^2/2\pi i\big)
$$
\end{cor}

\section{Bochner-Martinelli currents and Vogel cycles}\label{storskurk}

For a tuple $f=(f_0,\ldots, f_m)$ of holomorphic functions and $\beta=[\beta_0:\ldots:\beta_m]\in\P^m$ we write 
$\beta\cdot f:= \beta_{0} f_0 + \cdots +\beta_{m} f_m$. Note that $M^{\beta \cdot f}$ only depends on $\beta\in\P^m$ 
and not on the choice of homogeneous coordinates. 
Our first result in this section relates Lelong numbers of
Bochner-Martinelli currents to multiplicities of Vogel cycles.

\begin{thm}\label{meanvaluethmlelong} Let $f=(f_0,\ldots,f_m)$ be a tuple of holomorphic functions in $X$, 
pick $x\in X$, and let $Z=\{f=0\}$. 
Then for $k\geq 0$, and a generic choice of $\alpha=(\alpha_1,\ldots, \alpha_k)\in (\P^m)^k$, 
\begin{equation}\label{hutt3}
\ell_x\big({\bf 1}_Z 
[\alpha_k\cdot f]\w\cdots\w [\alpha_1\cdot f]\big)=\ell_x (M^f_k)\;.
\end{equation}
Here the current on the left hand side of \eqref{hutt3} should be interpreted as ${\bf 1}_Z$ if $k=0$. 
\end{thm}

Assume that $f=(f_0,\ldots,f_m)$ generates the ideal sheaf $\J$. Then 
$\alpha_1\cdot
f,\ldots, \alpha_n\cdot f$ is a Vogel sequence for
a generic choice of $\alpha=(\alpha_1,\ldots, \alpha_n)\in(\P^m)^n$. By
Theorem~\ref{meanvaluethmlelong} and Proposition~\ref{produktprop}, 
\begin{equation}\label{mintgron}
\mult_x V^{\alpha\cdot f}_k = \ell_x (M^f_k)
\end{equation}
and by Proposition ~\ref{remarkheldag} the right hand side only depends
on $\J$ and not on the particular choice of
generators $f$; in fact, it only depends on the integral closure of
$\J$. 
Thus this gives an independent proof of Gaffney-Gassler's result, \cite[Section~2]{GG}, that the
multiplicities of Vogel cycles $V^h$ are independent of $h$ for
generic $h$, which guarantees that the Segre numbers are
well-defined. Also, \eqref{ormens} immediately follows from
\eqref{mintgron}.


\begin{proof}
Choose a normal modification $\pi\colon \widetilde X\to X$ such that $\J(f)\cdot\O_{\widetilde X}$ is principal; 
we will use the notation from Section \ref{bmcurrents}. Assume moreover that the pullback of the maximal ideal at $x$ 
is principal, and let $D_\xi$ and $\omega_\xi$ be the corresponding divisor and Chern form, 
obtained from a tuple $\xi$ that defines the maximal ideal at $x$. 

Let $W$ be any irreducible subvariety of $\widetilde X$. Since $f'$ is nonvanishing on  $W$ it follows that
for $\beta$ outside a hypersurface in $\P^m$, the section $\beta\cdot f'$ is not vanishing identically on $W$.
By induction it follows that there is a  Zariski-open dense subset $A\subset (\P^m)^n$ such that
for each $\alpha=(\alpha_1,\ldots, \alpha_n)\in A$, the sequence
$\alpha_1\cdot f',\ldots ,\alpha_n \cdot f'$ is a geometrically regular sequence on each component
of $\widetilde X$, $|D_f|$, $|D_\xi|$, and on the support of $[D_\xi]\w[D_f]$.


Since $\pi^*(\alpha_\ell\cdot f)=f^0\, \alpha_\ell\cdot f'$, 
we have that
$
[\alpha_\ell\cdot f]=
\pi_*\big([D_f]+[\alpha_\ell\cdot f']\big)
$
and if $\alpha\in A$, in light of \eqref{rakneregel}, thus 
$$
[\alpha_2\cdot f]\w [\alpha_1\cdot f]=
\pi_*\big([D_f]\w [\alpha_1\cdot f']+[\alpha_2\cdot f']\w[\alpha_1\cdot f']\big).
$$
By induction,
\begin{multline}\label{ormbo}
[\alpha_k\cdot f]\w\cdots \w[\alpha_1\cdot f]=\\
\pi_*\big([D_f]\w[\alpha_{k-1}\cdot f']\w\cdots\w[\alpha_1\cdot f']+[\alpha_k\cdot f']\w\cdots
\w[\alpha_1\cdot f']\big),
\end{multline}
and so
\begin{equation}\label{ormbo2}
{\bf 1}_Z [\alpha_k\cdot f]\w\cdots \w[\alpha_1\cdot f] = 
\pi_*\big([D_f]\w[\alpha_{k-1}\cdot f']\w\cdots\w[\alpha_1\cdot f']\big) 
\,.
\end{equation}
Here we have used that ${\bf 1}_{D_f} [\alpha_k\cdot f']\w\cdots
\w[\alpha_1\cdot f']$ vanishes by Lemma ~\ref{polkagris}, 
and that 
\begin{equation}\label{stod}
{\bf 1}_Z\, (\pi_*\tau)=\pi_*({\bf 1}_{D_f}\tau). 
\end{equation}

For $k=0,1$,  \eqref{hutt3} follows from \eqref{MAlikhet0}, \eqref{data2} and \eqref{ormbo2}; 
in fact, the currents in \eqref{hutt3} coincide in these cases. 
Let us now assume that $k\geq 2$. 
We claim that there is a normal current $\A_k$  such that 
\begin{equation}\label{ormbo3}
dd^c\A_k=[D_\xi]\w\omega_\xi^{n-k-1}\w [D_f]\w\big(\omega^{k-1}_f-[\alpha_{k-1}\cdot f']\w\cdots\w[\alpha_1 
\cdot f']\big).
\end{equation}
For $\ell=1,\ldots, k$, $\log|\alpha_\ell\cdot f'|^2$ defines a singular metric
on $L^{-1}$ with first Chern form $[\alpha_\ell\cdot f']$, cf., ~\eqref{omegadef}, and 
thus $[\alpha_\ell\cdot f']$ is $dd^c$-cohomologous to $\omega_f$. More precisely, 
$c_\ell:=\log(|f'|^2/|\alpha_\ell\cdot f'|^2)$ is a global current on $\widetilde X$ and 
$\omega_f-[\alpha_\ell\cdot f']=dd^c  c_\ell$. 
Now, let 
$$
\A_k:=[D_\xi]\w\omega_\xi^{n-k-1}\w[D_f]\w
\sum_{\ell=1}^{k-1} \omega_f^{k-\ell-1}\w c_\ell\w [\alpha_{\ell-1}\cdot f']\w\cdots\w[\alpha_1\cdot f']\;. 
$$
Then $\A_k$ is normal. Since $\alpha_\ell\cdot f'$ does not vanish identically on any irreducible component
of (the support of) $[D_\xi]\w[D_f]\w[\alpha_{\ell-1}\cdot f']\w\cdots\w[\alpha_1\cdot f']$ it follows
from  Lemma~\ref{snittlemma} and the discussion after the proof of
that lemma that \eqref{ormbo3} holds.
From (the proof of) Proposition~\ref{heldag}  and \eqref{ormbo2} we get 
\begin{equation}\label{mvthm0}
dd^c\pi_*(\A_k) = 
\big (\ell_x (M_k^f)- \ell_x ({\bf 1}_Z [\alpha_k\cdot f] \w 
\cdots \w [\alpha_1 \cdot f]) \big ) [x] .
\end{equation}
On the other hand, $\pi_*\A_k$ is a normal $(n-1,n-1)$-current, 
and so since it has support at $x$, 
it vanishes according to Lemma \ref{polkagris}. 
\end{proof}

Our next result concerns mean values of Bochner-Martinelli
currents. In particular, it says that $M^f$
can be represented  a mean value of Vogel cycles.

\begin{thm}\label{meanvaluethm}
Assume that $f=(f_0,\ldots,f_m)$ is a tuple of holomorphic functions on $X$.  
Then 
\begin{equation}\label{hutt2}
M^f_k= \int_{\alpha\in(\P^m)^k}
{\bf 1}_Z [\alpha_k\cdot f]\w\cdots\w [\alpha_1\cdot f]
\end{equation}
where $Z=\{f=0\}$. 
Moreover, if  $\nu\ge\min(m+1,n+1)$, then
\begin{equation}\label{hutt1}
M^f=
\int_{\alpha=(\alpha_1,\ldots, \alpha_\nu)\in(\P^m)^{\nu}}M^{\alpha_{\nu}\cdot f}\w\cdots\w M^{\alpha_1\cdot f}.
 \end{equation}
\end{thm}


For the  proof we will use  the following lemma which is a simple variant
of Crofton's formula that should be well-known so we omit the proof,
see also \cite{aswy}.

\begin{lma}\label{snurr}
If  $\phi$ is a non-vanishing holomorphic $(m+1)$-tuple on $X$, 
then, in the sense of currents, 
$$
\int_{\beta\in\P^m}[\beta\cdot\phi]d\sigma(\beta)=dd^c\log|\phi|^2,
$$
where $d\sigma$ is the normalized Fubini-Study metric.
\end{lma}

\begin{proof}[Proof of Theorem~\ref{meanvaluethm}]
We use the same notation as in the  proof of Theorem~\ref{meanvaluethmlelong}.
In view of Lemma~\ref{snurr} and \eqref{omegadef} we have that
\begin{equation}\label{plats}
\int_{\beta\in\P^m}[\beta\cdot f']d\sigma(\beta)=\omega_f.
\end{equation}
Since all currents are positive we can apply Fubini's theorem and get 
\eqref{hutt2} from \eqref{ormbo2} by repeated use of \eqref{plats},
cf., \eqref{data2}.

We now prove \eqref{hutt1}. By \eqref{MAlikhet0} and \eqref{data2},  
$$
M^{\alpha_\ell\cdot f}=M_0^{\alpha_\ell\cdot f}+M_1^{\alpha_\ell\cdot f}=
{\bf 1}_{\alpha_\ell\cdot f}+[\alpha_\ell\cdot f].
$$
As in the proof of Proposition~\ref{vogelex} we get, for generic $(\alpha_1,\ldots,\alpha_{\nu})\in (\P^m)^{\nu}$, that
\begin{equation*}
M^{\alpha_{\nu}\cdot f}\w\cdots\w M^{\alpha_1\cdot f} =
\sum_{j=0}^{\nu} {\bf 1}_Z [\alpha_j\cdot f]\wedge \cdots \wedge [\alpha_1\cdot f] +
{\bf 1}_{X\setminus Z} [\alpha_{\nu}\cdot f]\wedge \cdots \wedge [\alpha_1\cdot f].
\end{equation*}
Moreover, it follows from \eqref{ormbo} and \eqref{ormbo2} that
\begin{equation*}
{\bf 1}_{X\setminus Z} [\alpha_{\nu}\cdot f]\wedge \cdots \wedge [\alpha_1\cdot f]=
\pi_*\big([\alpha_{\nu}\cdot f']\wedge \cdots \wedge [\alpha_1\cdot f']\big).
\end{equation*}
Hence, using \eqref{hutt2} and Lemma ~\ref{snurr}, we conclude that 
\begin{multline*}
\int_{\alpha=(\alpha_1,\ldots, \alpha_\nu)\in(\P^m)^{\nu}} 
M^{\alpha_{\nu}\cdot f}\w\cdots\w M^{\alpha_1\cdot f}
= M^f + \pi_* (dd^c\log|f'|^2)^{\nu}=M^f;
\end{multline*}
indeed, $(dd^c\log|f'|^2)^{\nu}=0$ since $\nu\geq \min (m+1,n+1)$. 
\end{proof}

By arguments as in the proof of  Theorem \ref{meanvaluethm} one can
check that 
\begin{equation}\label{skomakare}
 \int_{\alpha\in(\P^m)^k}
{\bf 1}_{X\setminus Z}[\alpha_k\cdot f]\w\cdots\w [\alpha_1\cdot f]=
{\bf 1}_{X\setminus Z}(dd^c\log|f|^2)^k.
\end{equation}



\section{Proof of the  generalized King formula (Theorem~\ref{genking})}\label{sectiongenking}
Let $X$ and $\J$ be as in Theorem \ref{genking} and let $Z$ be the variety of $\J$.   
The  {\it (Fulton-MacPherson) distinguished varieties}  
associated with $\J$ are defined in the following way, cf., \cite{fult}:
Let $\nu\colon X^+\to X$ be the normalization of the blow-up 
of $X$ along $\J$ and let $E$ be the exceptional divisor of $\nu$. Then $Z_j\subset X$ is a 
distinguished variety if it is the image under $\nu$ of an irreducible component of $E$. 
Let $Z^k_j$ be the distinguished varieties of codimension ~$k$. 
Also, we define the irreducible components of $X$ contained in $Z$ to be distinguished varieties (of codimension $0$).
 
\smallskip

Let us first consider the case $k=0$. By \eqref{uppdelning} we may assume that $X$ is irreducible. Then either $\J=(0)$ 
or $Z$ is a proper subvariety of $X$. In the first case $M_0^f=M_0^0=\1_X$ and if $h$ is a Vogel sequence of $\J$, 
then necessarily $h=(0,\ldots, 0)$ and so $V^h=V^h_0=X$. In the second case $M_0^f=0$ and if $h$ is a Vogel sequence 
of $\J$, then $V^h_0=X^{Z}_0=0$, since $X\not\subset Z$. It follows that Theorem \ref{genking} holds for $k=0$.

Next, consider the case $k\geq 1$. 
Let  $\pi\colon \widetilde X\to X$ be a normal modification such that $\J\cdot \mathcal O_{\widetilde X}$ principal. 
We use the notation from Section~\ref{bmcurrents}, so that $M^f_k=\pi_*([D]\w \omega^{k-1}_f)$, where $D= D_f$. 
Moreover, we let $D^k$ denote the components of $D$ that are mapped to sets
of  codimension $k$ in $X$. 
Note that $D=D^p+\ldots +D^n$, if  $p=\codim Z$.

If $\ell>k$, then $\pi_*([D^\ell]\w\omega_f^{k-1})$
is a positive closed $(k,k)$-current  with support on a variety of codimension $\ell>k$, 
and hence it must vanish in view of Lemma~\ref{polkagris}. 
Thus
\begin{equation}\label{boll}
M^f_k=S^f_k+N^f_k,
\end{equation}
where 
\begin{equation}\label{god}
S^f_k =\pi_*\big([D^k]\w  \omega_f^{k-1}\big), \quad
N^f_k=\pi_* \Big(\sum_{\ell<k} [D^\ell]\w  \omega_f^{k-1}\Big).
\end{equation}
Note that $M_k^f=0$ for $k<p$ and $N_p^f=0$. 
We claim that \eqref{boll} is the Siu decomposition of $M^f_k$, cf., Section~\ref{curranal}.
By Lemma~\ref{polkagris},
$S^f_k$ is the  Lelong current of a cycle of codimension $k$, so it is enough to show that $N^f_k$ does 
not carry any mass on varieties of codimension $k$.  
Let $W\subset X$ be such a  variety. By \eqref{stod}, 
\begin{equation}\label{world}
{\bf 1}_W \pi_*([D^\ell]\w\omega_f^{k-1})=
\sum_j\pi_*({\bf 1}_{\pi^{-1}W} [D^\ell_j]\w\omega_f^{k-1}),
\end{equation}
where $D_j^\ell$ are the irreducible components of $D^\ell$.  Then, since $\ell<k$, 
$\pi^{-1}(W)$ does not contain any component $D_j^\ell$, 
thus  each term in the right hand side of \eqref{world} vanishes,  and thus the claim follows. 

Since \eqref{boll} is the Siu decomposition of $M_k^f$, it follows that $S^f_k$ is independent of 
$\pi:\widetilde X\to X$. If we take $\pi$ to be the normalization of the blow-up of $\J$, we see 
that the $Z_j^k$ in \eqref{king} has to be among the distinguished varieties of $\J$. By 
Proposition~\ref{heldag} (for $r=1$), the Lelong number of $M_k^f$ is an integer at each point, and 
since the Lelong number of $N_k^f$ generically vanishes on each $Z_j^k$, we conclude that the $\beta_j^k$ and $\ell_x(N_k^f)$ 
are integers. 
That $\ell_x(N_k^f)$ is an integer can also be seen directly by copying the proof of Proposition~\ref{heldag}. 
Moreover, cf.,  Proposition~\ref{remarkheldag}, $\beta_j^k$ and $\ell_x(N_k^f)$ only depend on the integral closure of $\J$ at $x$.

We shall now see that the coefficients $\beta_j^k$ of the distinguished varieties are, in fact, $\geq 1$, following
the proof of Corollary 5.4.19, in \cite{Lazar}. 
The blow-up $\pi_{\J}\colon {\rm Bl}_{\J} X\to X$
of $X$ along  $\J$ can be seen as the subvariety of $X\times\P^m_t$
defined by the equations $t_jf_k-t_kf_j=0$, where $0\leq j<k\leq m$. Moreover, the line bundle
associated with the exceptional divisor
is the pullback of $\Ok_{\P^m_t}(-1)$ from $\P^m$ to ${\rm Bl}_{\J} X$, so 
$\omega_t=dd^c\log|t|^2$ represents minus its first Chern class.
This form is strictly  positive on the fibers of $\pi_\J$, and 
since the normalization   
$X^+\to {\rm Bl}_\J X$ is a finite map, the pullback $\omega$ of $\omega_t$ to 
$X^+$ remains strictly positive on the fibers of $\nu\colon X^+\to X$ as well. Let $E_j$ be one of 
the irreducible component of the exceptional divisor of $\nu$. We conclude that $\nu_*([E_j]\w\omega^{k-1})$ is 
a positive integer times $[Z_j^k]$, where $Z_j^k:=\nu(E_j)$. On the other hand, this current is unaffected
if we replace $\omega$ by $\omega_f$ since these two forms are  first Chern forms of the same line bundle.
It follows that  $\beta_j^k\geq 1$.

We saw in the discussion after Theorem ~\ref{meanvaluethmlelong} 
that $\ell_x(M_k^f)$ is equal to the $k$:th Segre number of $\J$ at $x$. 
Next, we show that the fixed Vogel components of $\J$ are precisely the $S_k^f$. Fix a point $x\in X$. 
As in proof of Theorem~\ref{meanvaluethmlelong} we can construct, for $k\geq 1$ and a generic $\alpha \in (\P^m)^n$, 
a normal current $\A_k$ with support on $|D^k|$ such that 
$$
dd^c \A_k =[D^k]\w([\alpha_{k-1}\cdot f']\w\cdots\w[\alpha_1\cdot f']-\omega_f^{k-1}).
$$
Now $\pi_* \A_k$ is a normal $(k-1,k-1)$-current with support on $\bigcup _jZ^k_j$, and thus it 
vanishes by Lemma \ref{polkagris}. It follows that 
$\pi_*\big([D^k]\w[\alpha_{k-1}\cdot f']\w\cdots\w[\alpha_1\cdot f'])=S^f_k 
$
and hence $S^f_k$ occurs in a generic Vogel cycle at $x$, meaning that $S^f_k$ is a fixed Vogel cycle.  
On the other hand, the cycles 
\begin{equation}\label{april}
\pi_*(\sum_{\ell<k}[D^\ell]\w[\alpha_{k-1}\cdot f']\w\cdots\w[\alpha_1\cdot f'])
\end{equation}
must be moving. Indeed, by (the proof of) Theorem \ref{meanvaluethm}, taking mean values of \eqref{april} over all 
$\alpha\in (\P^m)^k$, we get the current $N_k^f$, which carries no mass on any variety of codimension $k$, as seen above. 

By arguments as in the proof of Theorem \ref{meanvaluethmlelong} one shows that for a generic choice of $\alpha\in (\P^m)^k$, 
\begin{equation}\label{loud}
\ell_x(\1_{X\setminus Z} [\alpha_k\cdot f]\w\cdots\w [\alpha_1\cdot f]) =
\ell_x (\1_{X\setminus Z} (dd^c\log |f|^2)^k)
\end{equation}
cf. \eqref{skomakare}. However, it follows from Proposition~\ref{produktprop} that the left hand side of \eqref{loud} 
is equal to $m_k(x)$. 
This concludes the proof of Theorem~\ref{genking}.

\begin{remark}
One can see more directly that only the distinguished varieties occur
in $S^f_k$ if $S_k^f$ is defined by \eqref{god} from an arbitrary normal modification $\pi\colon \widetilde X\to X$. 
To begin with, $\pi$ factors over $\nu$, i.e., there exists a modification $\widetilde \nu\colon \widetilde X \to X^+$ such 
that $\pi=\nu \circ \widetilde \nu$. 
If $\omega_+$ is the form associated with $\J \cdot \mathcal O_{X^+}$ in $X^+$, then $\tilde\nu^* \omega_+=\omega_f$. 

Let $D^k_j$ be an irreducible component of the divisor $D^k$. 
Since $|D^k_j|\subset \pi^{-1}(Z)$, it follows that $\widetilde\nu( |D_j^k|)$  is
contained in one of the components $E_j$ of $E$ in $X^+$. If 
$\widetilde\nu(|D^k_j|)$ has codimension $\ge 1$ in $E_j$, then 
$\widetilde\nu_*([D^k_j]\w\omega_f^{k-1})=(\widetilde\nu_*[D^k_j])\w \omega_+^{k-1}$ vanishes by Lemma~\ref{polkagris}. 
Hence
$\pi_*([D^k_j]\w\omega^{k-1})=\nu_*\tilde\nu_*([D^k_j]\w\omega^{k-1})$ vanishes  unless $\widetilde\nu (|D_j^k|)=E_j$, 
in which case $\pi(|D_j^k|)$ is a distinguished variety.
\end{remark}

Assume that $f_0,\ldots, f_{p-1}$ is a regular sequence.  From the theory for proper intersections
we know that 
$$
[f_{p-1}]\w\cdots\w[f_0]=\sum\beta_j[Z_j]
 $$
where $Z_j$ are the irreducible components of $Z=\{f=0\}$, and that the intersection only 
depends on the ideal generated by the $f_j$, cf., Remark ~\ref{pia}. 
In particular, the right hand side 
is unaffected if we replace $f_j$ by $\alpha_j\cdot f$ for generic $\alpha_j$. From 
\eqref{hutt2} we conclude that
\begin{equation}\label{svans}
M^f_p= [f_{p-1}]\w\cdots\w[f_0].
\end{equation}

\begin{proof}[Proof of Corollary~\ref{smultron}]
If $Z$ is smooth, then locally there are coordinates $(z,w)$ so that 
$Z=\{w_1=\cdots=w_p=0\}$. In view of \eqref{svans} we have that 
$$
M_p^w=[w_p]\w\cdots\w[w_1]=[Z],
$$
and hence $\ell_x (M^w_p)=1$ for $x\in  Z$. If $Z$ is reduced and  $f$ generates the ideal $\J_Z$,  
therefore  $\ell_x (M^f_p)=1$
for all smooth  points $x\in Z$ in view of Proposition~\ref{remarkheldag}. From Theorem~\ref{genking} we know that
$M^f_p=\sum\beta_j[Z_j]$. Since the smooth points are dense, we conclude that $\beta_j=1$ for each $j$.
\end{proof}

\section{The minimality property}\label{minimal}


Recall that the \emph{lexicographical order} on $\R^N$ is a total order, defined by 
$(x_1,\ldots, x_N)\leq_{\lex} (y_1,\ldots, y_N)$ if there is an $1\leq\ell\leq N$ such that $x_i=y_i$ for $i\leq \ell$ 
and $x_\ell<y_\ell$. We let  $\min_\lex$ denote the minimum with respect to the lexicographical order. 
We will now give a proof that Tworzewski's extended index of
intersection coincides with the list of Segre numbers. 
More precisely we will prove: 



\begin{thm}\label{purra}
Let $\J$ be a coherent ideal sheaf on $X$ and let $e(x)$ be the list of
associated Segre numbers at $x$. Then 
\eqref{mintlex} holds, 
where the $\min_{\lex}$ is taken over all Vogel sequences $h$ of ideals with the same integral 
closure as $\J_x$. 

Moreover, if $f$ is a tuple of generators of $\J$ 
(or any ideal with the same integral clousure as $\J$) then it suffices to take 
the $\min_{\lex}$ over all Vogel sequences of the form $\alpha\cdot f=(\alpha_1\cdot f,\cdots,\alpha_n\cdot f)$, 
where $\alpha =(\alpha_1,\ldots,\alpha_n)\in (\P^m)^n$. 
\end{thm}

As mentioned in the introduction, Theorem~\ref{purra} is
known in the  case when
 $\J$ is the pullback to $X$ of the radical sheaf of 
a smooth manifold $A$ in some ambient space.
%
%
For the proof of Theorem~\ref{purra}
we will need the following result; if  $Z$ is smooth this is precisely Theorem~3.4 in \cite{T}.


\begin{prop}\label{prop6-2} Assume that $(W_j)_{j\in \N}$ and $W$ are subvarieties of $X$ of pure dimension 
such that $\lim_{j\rightarrow \infty} [W_j]=[W]$ as currents on $X$. Let $Z$ be a fixed subvariety of $X$, 
let $x$ be a fixed point in $Z$, and assume that
\begin{equation}\label{skata}
\ell_x( {\bf 1}_Z[W])\le  \ell_x({\bf 1}_Z[W_j]).  
\end{equation}
for all $j$. 
Then there is a \nbh $U$ of $x$ in $X$, in which 
$
\lim_{j\rightarrow \infty} ({\bf 1}_Z[W_j]) = {\bf 1}_Z[W]$ and 
$\lim_{j\rightarrow \infty} ({\bf 1}_{X\setminus Z}[W_j])= {\bf 1}_{X\setminus Z}[W]
$.
\end{prop}

\begin{proof}
Since the currents $[W_j]$ are positive and locally uniformly bounded, 
so are the currents ${\bf 1}_Z[W_j]$. 
Thus, there is a subsequence $({\bf 1}_Z[W_{j_k}])_{k\in \N}$ converging to a 
positive closed current with support 
on $W\cap Z$. By Lemma ~\ref{polkagris} this current is the integration current $[V]$ for some 
effective cycle $V$ (with possibly real coefficients).
Since $[W_j]-{\bf 1}_Z[W_j]$ is positive, so is $[W]-[V]=\lim_k ([W_{j_k}]-{\bf 1}_Z[W_{j_k}])$,  
and since $|V|\subset |Z|$, it follows that
\begin{equation}\label{ormet}
[V]={\bf 1}_Z[V]\le {\bf 1}_Z [W].
\end{equation}
By \eqref{skata} and semicontinuity, \eqref{limsup}, we have that
$$
\ell_x ({\bf 1}_Z[W])\le \limsup_k
 (\ell_x ({\bf 1}_Z[W_{j_k}]))\le \ell_x ([V])\le \ell_x ({\bf 1}_Z [W]).
$$
Thus $\ell_x ({\bf 1}_Z[W])=\ell_x ([V]$), and combined with \eqref{ormet} 
and the fact that $V$ and $W$ are effective cycles, 
it  follows that $[V]={\bf 1}_Z[W]$ in some \nbh of $x$.

Since each subsequence of $({\bf 1}_Z[W_j])_{j\in \N}$ has a subsequence that tends to ${\bf 1}_Z[W]$, 
it follows that 
$\lim_{j\rightarrow \infty} ({\bf 1}_Z[W_j])={\bf 1}_Z[W]$. 
The last statement follows by complementarity. 
\end{proof}


\begin{proof}[Proof of Theorem \ref{purra}] 
Since each Vogel sequence $h$ can be realized as $\alpha\cdot f$ for some choice of $f$ and $\alpha$, 
it is easy to check that the first statement follows from the second one. 
Let $f$ be a tuple of generators of $\J$. 
By definition, $e(x)=\mult_xV^{\alpha\cdot f}$ for almost all
$\alpha$, and thus it is enough to  prove that 
$e(x)\leq_{\lex} \min_{\lex}\mult_x V^{\alpha\cdot f}$
if $\alpha\cdot f$ is a Vogel sequence. 

Suppose that $e(x)\not\leq_{\lex}\min_{\lex}\mult_x V^{\alpha\cdot f}$. 
Then there is an $r$ and a Vogel sequence $\alpha\cdot f$ 
such that $e_k(x)=\mult_x V^{\alpha\cdot f}_k$ for $k\le r-1$ but
$\mult_x V^{\alpha\cdot f}_r < e_r(x)$. Since $\alpha\cdot f$ is a Vogel sequence of $\J$ for a 
generic choice of $\alpha$, we can choose $(\alpha^j)_{j\in \N}$ in $(\P^m)^n$ such that $(\alpha^j)_{j\in \N}
\to \alpha$ and such that $\alpha^j\cdot f$ is a Vogel sequence of $\J$ for each $j$, and moreover, 
by Theorem \ref{meanvaluethmlelong}, such that $\mult_x V^{\alpha^j\cdot f}=e(x)$. 
It then follows that, for $k\leq r$,  
\begin{equation}\label{puta}
\ell_x({\bf 1}_Z[\alpha_k\cdot f]\w\cdots\w[\alpha_1\cdot f])
 \le e_k(x) = \ell_x({\bf 1}_Z[\alpha^j_k\cdot f]\w\cdots\w[\alpha^j_1\cdot f]).
\end{equation}

We claim that
\begin{equation}\label{puta2}
\lim_{j\rightarrow \infty} [\alpha^j_k\cdot f]\w\cdots\w[\alpha^j_1\cdot f] = 
[\alpha_k\cdot f]\w\cdots\w[\alpha_1\cdot f]
\end{equation}
for $k\le r$. For instance by \cite[Chapter~2, Corollary 12.3.4]{Ch}, \eqref{puta2} holds for $k=1$. Assume now that it holds for
$k<r$. Then by \eqref{puta} and Proposition~\ref{prop6-2}, 
\begin{equation}\label{puta2bis}
\lim_{j\rightarrow \infty} ({\bf 1}_{X\setminus Z}[\alpha^j_k\cdot f]\w\cdots\w[\alpha^j_1\cdot f])= 
{\bf 1}_{X\setminus Z}[\alpha_k\cdot f]\w\cdots\w[\alpha_1\cdot f]. 
\end{equation} 
Since $\alpha^j\cdot f$ and $\alpha\cdot f$ are Vogel sequences, the currents in \eqref{puta2bis} intersect properly with 
$[\alpha^j_{k+1}\cdot f]$ and $[\alpha_{k+1}\cdot f]$, respectively. In light of  \cite[Chapter~2, Corollary 12.3.4]{Ch} or 
\cite[Theorem~3.6]{T},  \eqref{puta2} holds for $k+1$, and the claim follows by induction.

Proposition~\ref{prop6-2} and \eqref{puta} imply that 
\begin{equation}\label{bicicletta}
\lim_{j\rightarrow \infty} 
({\bf 1}_{Z}[\alpha^j_r\cdot f]\w\cdots\w[\alpha^j_1\cdot f]) = 
{\bf 1}_{Z}[\alpha_r\cdot f]\w\cdots\w[\alpha_1\cdot f].
\end{equation}
By semicontinuity, \eqref{limsup}, the Lelong number of the right hand side of \eqref{bicicletta}
is greater than or equal to the Lelong number of ${\bf 1}_{Z}[\alpha^j_r\cdot f]\w\cdots\w[\alpha^j_1\cdot f]$.
Thus, $\mult_x V^{\alpha\cdot f}_r\geq e_r(x)$, which gives a contradiction. Hence $\min_{\lex}\mult_x V^{\alpha\cdot f}= e(x)$. 
\end{proof}

Given a positive closed current $v$, 
we define $\ell_x(v):=(\ell_{x0},...,\ell_{xn})$, where $\ell_{xk}$ denotes the Lelong number 
at $x$ of the component of $v$ of bidegree $(k,k)$. 
If $v$ and $w$ are positive and closed, 
we let
$
v\le_{x} w
$
mean that $\ell_x (v)\le_{\lex} \ell_x (w)$, 
and $v=_xw$ means that  $\ell_x (v)=\ell_x (w)$. 
Observe that if $h$ is a Vogel sequence of an ideal $\J_x$, then the zero sets of $h$ and $\J_x$ coincide. 
If $f_1,\ldots, f_n$ is Vogel sequence of an ideal $\J_x$, then in
view of Theorems \ref{genking} and \ref{purra} and Proposition \ref{vogelex}, $M^f\le_{x} M^{f_n}\w\ldots \w
M^{f_1}$. In fact we have: 
\begin{prop}\label{theoremorder} 
Let $f_1,\ldots, f_\m$ be a sequence of elements in $\mathcal O_{X,x}$
and let $f=(f_1,\ldots,f_s)$. 
Then
\begin{equation}\label{puta3}
M^f\le_{x} M^{f_\m}\w\ldots \w M^{f_1}. 
\end{equation}
\end{prop}

\begin{proof}
Let $Z:=\{f=0\}$.  
In order to prove \eqref{puta3}, we proceed by induction on the number $\m$ of functions. Clearly \eqref{puta3} 
holds for $\m=1$, so assume that it holds for $\m-1$ instead of $\m$. Let $\widetilde f:=(f_2,\ldots, f_{\m})$. 
By \eqref{uppdelning} we may assume that $X$ is irreducible and that
$f_1$ does not vanish identically on $X$, so that $M^{f_1}=M^{f_1}_1=[f_1]$; 
otherwise $M^{f_1}=M_0^{f_1}={\bf 1}_X$ and $M^f=M^{\widetilde f}$ 
and we are back in the case $\m-1$. 

Let $[W]:=[f_1]$, and let $i_{W_j}:W_j\hookrightarrow X$ be the irreducible components of $W^{X\setminus Z}$. 
Theorem \ref{meanvaluethmlelong} asserts that for a generic choice of $\alpha\in (\P^{\m-2})^{n-1}$, $\alpha\cdot \widetilde f$ 
is a Vogel sequence\footnote{If $s=2$, then $\tilde{f}=f_2$ and $\mathbb{P}^0$ should 
be interpreted as $\{1\}$.} of $\J(i_{W_j}^*\widetilde f)$ and $
M^{\alpha_{n-1}\cdot \widetilde f}\w\cdots\w
M^{\alpha_1\cdot \widetilde f}
=_x M^{\widetilde f}
$
on each $W_j$, so that 
$
M^{\alpha_{n-1}\cdot \widetilde f}\w\cdots\w M^{\alpha_1\cdot \widetilde f}\w 
[W^{X\setminus Z}]
=_x M^{\widetilde f}\w [W^{X\setminus Z}], 
$
cf.\ the discussion after Lemma~\ref{polkagris}.
By the induction hypothesis 
$$
M^{\widetilde f}\w [W^{X\setminus Z}]\le_x M^{f_\m}\w\cdots\w M^{f_2}\w [W^{X\setminus Z}]. 
$$
In view of \eqref{jobb} and \eqref{jobb2}, since $[f_1]=[W^Z] + [W^{X\setminus Z}]$ and $\tilde{f}$ vanishes on $Z$,
we get 
\begin{equation}\label{hummer}
M^{\alpha_{n-1}\cdot \widetilde f}\w\cdots\w 
M^{\alpha_1\cdot \widetilde f}\w M^{f_1} 
\le_x M^{f_\m}\w\cdots\w M^{f_2}\w M^{f_1}.
\end{equation}
For a generic choice of $\alpha$, the sequence 
$f_1,\alpha_1\cdot \widetilde f, \ldots, \alpha_{n-1}\cdot\widetilde f$ is a Vogel sequence of $\J(f)$. 
Thus, by Theorem \ref{purra}, 
\begin{equation}\label{hummer2}
M^f\le_x M^{\alpha_{n-1}\cdot \widetilde f}\w\cdots\w 
M^{\alpha_1\cdot \widetilde f}
\w M^{f_1}.
\end{equation}
Combining \eqref{hummer} and \eqref{hummer2}, we get \eqref{puta3}. 

\end{proof}

\section{An invariance property}\label{invprop}


We have the following invariance property of Segre numbers. In the
setting when $\J$ is the pullback to $X$ of the radical sheaf of a
smooth manifold in some ambient space this was proved in \cite[Section~5]{R}.

\begin{prop}\label{laban}
Let $\J$ be an ideal sheaf on an analytic space $X$, let $\J'$ be the pullback of $\J$ to $X\times\C_w$ under the projection $(x,w)\mapsto x$, and
let $i\colon X\hookrightarrow X\times\C_w$ be the embedding $x\mapsto (x,0)$. Then
\begin{equation}\label{apa1}
e_k(\J',X\times\{0\},i(x))=e_k(\J,X,x)
\end{equation}
and
\begin{equation}\label{apa2}
e_{k+1}(\J'+(w),X\times\C_w,i(x))=e_k(\J,X,x).
\end{equation}
\end{prop}

For the proof we will use the following invariance of Bochner-Martinelli
currents. 

\begin{lma} \label{hoppsan}
Let $f$ be a tuple of holomorphic functions on $X$ and 
let $i:X\hookrightarrow X\times \C_w$   be the  embedding 
$x\mapsto (x,0)$.  
Then $M^{(f,w)}_0=0$ and
\begin{equation}\label{lma2}
M^{(f,w)}_{k+1}=i_* M^f_k , \quad k\ge 0.
\end{equation}
Moreover, if $W\subset X$ is an analytic variety, 
\begin{equation}\label{nytt}
M_k^{f\otimes 1}\w[W\times\{0\}]=i_*(M_k^f\w[W]).
\end{equation}
\end{lma}


If we consider $X$ as embedded in some larger analytic space $X'$ and
$i\colon X'\to X'\times\C_w$, $x\mapsto (x,0)$, then \eqref{lma2} reads
\begin{equation*}
M_{k+1}^{(f,w)}\w[X\times \C_w]=i_*(M^f_k\w [X]).
\end{equation*}
In particular, if $f=0$,
\begin{equation}\label{mednoll}
M_1^w\w[X\times \C_w]=i_*[X]=[X\times\{0\}].
\end{equation}

\begin{proof}
Let $z$ be local coordinates on $X$. Since $(z,w)\mapsto (f(z),w)$ does not vanish identically on $X\times \C_w$, 
it follows that $M^{(f,w)}_0=0$. 

Let us now prove \eqref{lma2}. First consider the case when $k=0$. By \eqref{uppdelning} we may assume that $X$ 
is irreducible. Then either $f\equiv 0$ on $X$ or the zero set of $f$ has at least codimension $1$ in $X$. In the first case 
$$
M^{(f,w)}_1=M^{w}_1=[w]=i_*1=i_*M^0_0=i_*M^f_0.
$$
In the latter case the zero set of $(f,w)$ has at least codimension $2$ on $X\times\C_w$, and and so both sides 
of \eqref{lma2} vanish by Lemma \ref{polkagris}.
Thus \eqref{lma2} holds for $k=0$.

Next let $\pi\colon \widetilde X\to X$ be a smooth modification such that $\J\cdot \mathcal O_{\widetilde X}$ 
is principal and moreover $f^0$ is locally a monomial; use the notation from Section \ref{bmcurrents}.
Observe that then 
$\pi\otimes {\rm id}_w\colon\widetilde X\times\C_w\to X\times\C_w$
is a smooth modification with the same properties.  
It follows that it is enough to prove \eqref{lma2} in case $X$ is smooth, 
$\J=(f^0)$ is principal and $f^0$ is (in local coordinates) a monomial.

In light of Section \ref{bmcurrents} we thus have to show that 
\begin{equation}\label{aptit}
(2\pi i)^{-1}\dbar(|f|^2+|w|^2)^\lambda\w \partial\log(|f|^2+|w|^2)\w
(dd^c\log(|f|^2+|w|^2))^k
\end{equation}
is equal to 
$[f^0]\w(dd^c\log|f'|)^{k-1}\w [w]=i_*M_k^f$ 
when $\lambda=0$. 
Indeed, at $\lambda=0$, \eqref{aptit} is equal to $M_{k+1}^{(f,w)}$. 
 Note that \eqref{aptit} is locally integrable for $\Re\lambda>0$.   
Moreover, if $\Re\lambda <1$, it is integrable in the $w$-direction and thus acts on forms that are 
just bounded in the $w$-direction. Since $M_{k+1}^{(f,w)}$ is of order zero and $\supp M_{k+1}^{(f,w)}\subset \{w=0\}$, 
it follows that to check the action of $M_{k+1}^{(f,w)}$ on test forms, it is enough to consider 
forms $\xi(z,w)=\tilde\xi(z)$, where $\tilde\xi(z)$ is any test form in $X$. 
However, after the (generically $1-1$) change
of variables  $f^0\omega=w$, 
so that $|f|^2+|w|^2=|f^0|^2(|f'|^2+|\omega|^2)$,  the action of \eqref{aptit} on $\xi$ is equal to 
$$
(2\pi i)^{-1}\int_{z,\omega}\dbar|f^0|^{2\lambda}(|f'|^2+|\omega|^2)^\lambda
\w\partial\log|f^0|^2\w
(dd^c\log(|f'|^2+|\omega|^2))^k\w \tilde\xi(z). 
$$
Taking $\lambda=0$, we get
\begin{equation}\label{arstid}
\int_z[f^0]\w\tilde\xi(z)\w\int_\omega(dd^c\log(|f'|^2+|\omega|^2))^k. 
\end{equation}
One can check that the inner integral in \eqref{arstid} is equal to
$(dd^c\log|f'|^2)^{k-1}$, which proves \eqref{lma2}.
Finally we prove \eqref{nytt}. 
Let $j:W\hookrightarrow X$. Then, using \eqref{jobb},
\[
M^{f\otimes 1}\w[W\times \{0\}]=i_*j_*M^{j^*i^* f\otimes 1}=
i_*j_*M^{j^*f}=i_*M^f\w[W].
\]
\end{proof}


\begin{proof}[Proof of Proposition \ref{laban}]
Since the pullback of $\J'$ 
to $X\simeq X\times\{0\}$ is just $\J$, \eqref{apa1} should
be clear. More formally:
Let $f$ be a tuple that defines the ideal sheaf $\J$ in $X$. 
Then $f\otimes 1$ defines $\J'$ in $X\times\C_w$ and
\begin{equation*}
e_k(\J',X\times\{0\},x)=
\ell_x(M^{f\otimes 1}_k\w[X\times\{0\}])=
\ell_x (M^f_k\w[X]) = e_k(\J,X,x),
\end{equation*}
where we have used \eqref{nytt} for the second equality. 
This  proves \eqref{apa1}.

To see \eqref{apa2} notice that $f,w$ defines $\J+(w)$ in $X\times\C_w$. 
Thus, using \eqref{jobb} and \eqref{lma2} we have 
\begin{equation*}
e_{k+1}(\J'+(w),X\times\C_w,(x,0))=\ell_x(M^{(f,w)}_{k+1}\w[X\times\C_w])=
\ell_x(M^f_k\w[X])=e_k(\J,X,x).
\end{equation*}
\end{proof}

\section{Local intersection numbers}\label{oxet}

Tworzewski's original motivation for introducing the extended index of
intersection was to understand intersection theory in the nonproper
case. 
Let $Z_1,\ldots,Z_r$  be subvarieties of a smooth manifold $Y$ that do not 
necessarily intersect properly.  A standard procedure to define an
intersection product $Z_1\cdots Z_r$ 
is to give some reasonable meaning to the intersection 
\begin{equation}\label{pop}
\Delta\cdot  Z_1\times\cdots\times Z_r,
\end{equation} where $i\colon Y\simeq \Delta\to Y\times\cdots \times Y$
is the diagonal in
$Y\times\cdots \times Y$. In this way one is reduced to the case of two varieties one of which is smooth. 

Now assume that $A, Z$ are subvarieties of $Y$, that $A$
is smooth, and (initially) that $Z$ has pure dimension. Let $\J_A$ denote the radical sheaf of $A$, and also, for
simplicity, the pullback of $\J_A$ to $Z$. 
Following Tworzewski 
we define {\it local intersection
  numbers} 
\begin{equation}\label{twortwor}
g_\ell(A, Z, x):=e_{\dim Z-\ell}(\J_A,Z,x)
\end{equation}
at $x$. If $Z=\sum_j\alpha_j Z_j$, where the $Z_j$ are pure dimensional, we set
$g_\ell(A, Z, x):=\sum_j\alpha_j g_\ell(A, Z_j, x)$. 
The change of
indices is made so that $\ell$ corresponds to the generic multiplicity
of components of {\it dimension} $\ell$ of Vogel cycles. 
We will use the notation $A\circ Z$ for these lists of
local intersection numbers, i.e., 
\begin{equation}\label{two}
A\circ Z (x) = (g_{\dim Z}(A, Z, x), \ldots, g_1(A, Z, x), g_0(A, Z, x)) 
\end{equation}

The local \emph{multiplicities of intersection} of general varieties $Z_1,\ldots,
Z_r$ are then 
\begin{equation}\label{lokalsnitt}
\epsilon_\ell(Z_1,\ldots,
Z_r;x)=g_\ell(\Delta,Z_1\times\cdots\times Z_r,(x,\ldots,x)),  
\end{equation}
cf.\ \cite[Section~6]{T}. 
We will write 
\begin{equation}\label{three}
Z_1\diamond \cdots \diamond Z_r(x)=(\epsilon_\nu(Z_1,\ldots, Z_r;
x),\ldots, 
\epsilon_1(Z_1,\ldots,  Z_r; x), \epsilon_0(Z_1,\ldots,  Z_r; x) ),  
\end{equation} 
where $\nu$ is the dimension of the set-theoretical intersection
$Z_1\cap\cdots\cap Z_r$. 

Note that the product $Z_1\diamond \cdots \diamond Z_r$ by definition is
commutative. It is also independent of the manifold $Y$ in the
following sense: If $\iota: Y\to \widetilde Y$ is an embedding of 
 $Y$ in a larger manifold $\widetilde Y$, then 
\begin{equation}\label{brooklyn} 
Z_1\diamond\cdots\diamond Z_r(x)= \iota(Z_1)\diamond\cdots\diamond
\iota(Z_r)\big (\iota(x)\big). 
\end{equation} 
This follows from Proposition ~\ref{laban}, see also
\cite[Section~5]{R}. 

The next result, which relates the two local intersections \eqref{two}
and \eqref{three}, we have not found in the literature. 


\begin{prop}\label{snuttefilt}
Assume that $A,Z$ are subvarities of a manifold $Y$, and that $A$ is
smooth. Then 
\[
A\diamond Z = A \circ Z. 
\]
\end{prop}

In particular, if $A$ and $B$ are smooth submanifolds of $Y$, since
$A\diamond B$ is commutative, it follows that $A\circ B=B\circ A$. 

\begin{proof}
Fix $x$ in $Y$. We may assume, without loss of generality,  that $Y=\C^N$. 
Choose local coordinates $z=(z',z'')$ on $\C^N$ so that $A=\{z'=0\}$, and 
local coordinates $(z,w)$ on $\C^N\times \C^N$. 
We will  show that 
\begin{equation}\label{forsta}
\epsilon_j(A,Z;x)=e_{\dim A+ \dim Z - j}(\J_\Delta,Z\times
A,x)=\ell_x(M^{z-w}_{\dim A+ \dim Z- j}\w [Z\times A]), 
\end{equation}
cf.\ \eqref{jobb}, 
coincides with 
\begin{equation}\label{andra}
g_j(A,Z,x)=e_{\dim Z - j}(\J_A,Z,x)=\ell_x(M^{z'}_{\dim Z- j}\w [Z]).
\end{equation}

Note that 
$M^{z-w}_k\w[Z\times A]=M^{(z',z''-w'')}_k\w[Z\times \{w'=0\}]$. 
Let $(z',z'', w', \eta'')$, where $\eta''=z''-w''$, be new coordinates on $\C^N\times\C^N$. Then  \eqref{lma2}  
implies that $M_{k+\dim A}^{(z',\eta'')}\w[Z\times \{w'=0\}]=i_*M^{z'}_{k}\w[Z\times \{w'=0\}]$, 
where $i:\C^{2N-\dim A}_{z',z'',w'}\hookrightarrow \C^{2N}_{z',z'',w',\eta''}$. 
Moreover, by \eqref{nytt}, $M_k^{z'}\w[Z\times\{w'=0\}]=j_*M^{z'}_k\w[Z]$, where 
$j:\C^N_{z',z''}\hookrightarrow \C^{2N-\dim A}_{z',z'',w'}$. 
Hence 
$$
M^{z-w}_{\dim A+\dim Z-j}\w[Z\times A]=i_*j_* M^{z'}_{\dim Z-j}\w[Z]
$$ 
and thus \eqref{forsta} is equal to \eqref{andra}.
\end{proof}

\begin{ex}\label{modem}
If $Z_j$ intersect properly, then
also \eqref{pop} is a proper intersection and it is well-known that 
$Z_1\cdots Z_r$ coincides with the intersection  \eqref{pop} (after identifying $Y\simeq \Delta$).
It follows that 
$$
\epsilon_\ell(Z_1,\ldots, Z_r;x)=\mult_x (Z_1\cdots Z_r)
$$
for $x\in |Z_1\cdots Z_r|$ and $\ell=\dim(Z_1\cdots Z_r)$,
and $0$ otherwise, 
cf. \cite[Theorem~6.5]{T}. 
In particular, if $Z$ is a subvariety of the smooth manifold
$A$, in view of \eqref{brooklyn}, 
\begin{equation}\label{likhet}
A\diamond Z (x)= (\mult_x Z,0,\ldots,0).  
\end{equation} 
\end{ex}

In the nonproper case the classical intersection product
$Z_1\cdots Z_r$ 
in the sense of Fulton, \cite{fult}, 
cannot represent the local intersection multiplicities in any reasonable
sense. Indeed,  in general $Z_1\cdots Z_r$ is a cycle of codimension $\codim Z_1+\cdots + \codim Z_r$,
determined modulo rational equivalence on $Z_1\cap\cdots\cap Z_r$. 



\smallskip

Tworzewski, \cite[Sections~5,6]{T}, proves that there is unique cycle
$Z_1\bullet\cdots \bullet Z_r$, that he calls the \emph{intersection
  product}, such that 
$$
\sum_k\mult_x (Z_1\bullet\cdots\bullet Z_r)_{k}=\sum_\ell
\epsilon_\ell(Z_1,\ldots,Z_r;x) 
$$
for each point $x$, where the index $k$ denotes the component of dimension
$k$. 
By definition this product  
respects the (sum of the) local intersection multiplicities, but
it does  not respect Bezout's formula in general. For instance,
the self-intersection in $\P^2$ of any smooth curve $C$ is just $C$ 
itself, and thus $\deg (C\bullet C)\neq (\deg C)^2$ unless $C$ is a
line.

\smallskip
In a forthcoming paper we will introduce,  in the case $Y=\P^n$,  a global current that represents, at each point, the local intersection multiplicities, and 
respects Bezout's formula, in a reasonable sense. It is obtained as the mean
value of various Vogel sequences, based on global  variants of 
the ideas in Section~\ref{storskurk} above.

\section{Examples}\label{exsection}


Let us start by some computations of Bochner-Martinelli currents and Segre
numbers. Our first example illustrates that the currents $M^f$ in
general depend on the set of generators $f$ although the Lelong numbers only
depend on (the integral closure of) the ideal generated by $f$, cf.\ Remark ~\ref{pia}. 


\begin{ex} Let us consider the primary ideal $(x)$ in $\C^2_{x,y}$. We
  know from Corollary~\ref{smultron} and Remark ~\ref{pia}, respectively, that 
$M^x_1=[x]$ and that $M^x_2=0$.  Let us now consider the pair $(x,xy)$ of generators for the same ideal.
We first consider the Vogel cycles obtained from generic linear combinations of these generators.
Since $Z=\{x=0\}$ and 
$
[\alpha_0 x+\alpha_1 xy]=[x(\alpha_0+\alpha_1y)]=[x]+[\alpha_0+\alpha_1y]
$
we have that 
$$
\1_Z [\alpha_0 x+\alpha_1 xy]=[x]
$$
as expected, since $[x]$ must be a fixed component in any Vogel cycle. A  simple computation yields that
$$
\1_Z [\beta_0 x+\beta_1 xy]\w[\alpha_0 x+\alpha_1 xy]=[x]\w [\alpha_0+\alpha_1 y]
$$
for generic choices of $\alpha$ and $\beta$, cf.\ Section ~\ref{lelongcartier}.  Thus the component of the Vogel cycle of codimension
$2$ is non-vanishing for generic $\alpha,\beta$. Taking mean values over $\P^1$ we get, cf., Theorem~\ref{meanvaluethm},
$$
M^{x,xy}_2=[x]\w\frac{dy\w d\bar y}{\pi(1+|y|^2)^2}.
$$
Here we have used that with the generic parametrization $\C\ni t\mapsto [-t,1]\in\P^1$, we  have 
$$
\int_{[\alpha]\in\P^1}[\alpha_0+\alpha_1 y]d\sigma(\alpha)=\int_{t\in\C}[y-t]\w\frac{dt\w d\bar t}{\pi(1+|t|^2)^2}=
\frac{dy\w d\bar y}{\pi(1+|y|^2)^2}.
$$
\end{ex}


Next we will discuss a simple example where a moving component
occurs. 

\begin{ex}\label{hak1}
Consider the tuple $f=t_3(t_1,t_2,t_3)=t_3 t$ in $X=\C^3_t$, with zero set $Z=\{t_3=0\}$. 
We will compute the Segre numbers $e_k(0)=e_k(\J(f),\C^3,0)$, $k=0,1,2,3$.
Let $\alpha\cdot f$ be a Vogel sequence of $\J(f)$ at $0$ of the 
form $\alpha_1\cdot f, \ldots, \alpha_3\cdot f$. 
Let us compute the corresponding Vogel cycle $V^{\alpha\cdot f}$. 
First note that $X^{X\setminus Z}_0=X_0=X$. 
Thus, by Proposition \ref{produktprop},  
$$[X_1]=M_1^{t_3(\alpha_1\cdot t)}=[t_3]+[\alpha_1\cdot t]=[X_1^Z]+[X_1^{X\setminus Z}].$$ 
Furthermore, using \eqref{tundra} and \eqref{rakneregel}, we get
$$
[X_2]=M_1^{t_3(\alpha_2\cdot t)}\w M_1^{t_3(\alpha_1\cdot t)}=
[t_3]\w[\alpha_1\cdot t]+[\alpha_2\cdot t]\w[\alpha_1\cdot t]=[X_2^Z]+[X_2^{X\setminus Z}]
$$
and
$$
[X_3]=M_1^{t_3(\alpha_3\cdot t)}\w M_1^{t_3(\alpha_2\cdot t)}\w M_1^{t(\alpha_1\cdot t)}=
([t_3]+[\alpha_3\cdot t])\w [\alpha_2\cdot t]\w[\alpha_1\cdot t]=2[0] = [X_3^Z],
$$
for a generic $\alpha$. 
Hence 
\begin{equation*}
[V^h]=[V^h_1]+[V^h_2]+[V^h_3]=[t_3]+[t_3]\w[\alpha_1\cdot t]+2[0]
\end{equation*}
and, in particular, 
$$
e_0(0)=0,\  e_1(0)=1,\  e_2(0)=1,\  e_3(0)=2.
$$ 
Observe that $V^h_1$ and $V^h_3$ are fixed, whereas $V^h_2$ is moving. 
A computation, using Theorem~\ref{meanvaluethm} and Lemma~\ref{snurr}, yields
$$
M^f_0=0,\  M^f_1=[t_3],\  M^f_2=[t_3]\w dd^c\log(|t_1|^2+|t_2|^2),
\  M^f_3=2[0].
$$
\end{ex}

The following simple lemma is useful for computations.
\begin{lma}\label{listigt} Let $X$ and $X'$ be two analytic spaces
of dimension $n$, let $\tau\colon X'\to X$ be a holomorphic map, 
and let $f$ be a tuple of holomorphic functions on $X$. 
Assume that  $\tau$ is proper, surjective, and generically $r$ to $1$. Then 
\begin{equation}\label{listigt1}
r M_k^f=\tau_* M_k^{\tau^*f}.
\end{equation}

Moreover, if $\xi$ is a tuple that defines the maximal ideal at $x\in X$, then the Segre numbers at $x$
associated with $\J=\J(f)$ on $X$ are given by 
\begin{equation}\label{listigt2}
e_k(x)=\frac{1}{r}\int_{X'} M^{\tau^*\xi}_{n-k}\w M^{\tau^* f}_k.
\end{equation}
\end{lma}

\begin{proof}
Since $\tau^* M_k^{f,\lambda}=M_k^{\tau^*f,\lambda}$  if $\Re\lambda\gg0$, we have that then 
$$
\int_X M_k^{f,\lambda}\w \psi =\frac{1}{r}\int_{X'} M_k^{\tau^*f,\lambda}\w\tau^*\psi
$$
for test forms $\psi$. Taking analytic continuations to $\lambda=0$, we get \eqref{listigt1}.
In view of Proposition~\ref{prop3-4} we have  
\begin{multline*}
e_k(x)=\ell_x M^f_k=\int_X M^{\xi,\lambda}_{n-k}\w M_k^{f,\lambda^2}\big|_{\lambda=0}=\\
\frac{1}{r}
\int_{X'} M^{\tau^*\xi,\lambda}_{n-k}\w M_k^{\tau^*f,\lambda^2}\big|_{\lambda=0}=
\frac{1}{r}\int_{X'} M^{\tau^*\xi}_{n-k}\w M^{\tau^* f}_k.
\end{multline*}
\end{proof}

In particular, it follows from Lemma~\ref{listigt} that 
\begin{equation}\label{lelongtal} 
\mult_x X=\int_X M^\xi_n=\frac{1}{r}\int_{X'} M_n^{\tau^*\xi}.
\end{equation}


\begin{ex}\label{cuspen}
Let $r,s$ be relatively prime integers and consider the cusp $X=\{ z_1^r-z_2^s=0\}$ in $\C^2_z$.
Since we have the parametrization $\tau\colon t\mapsto (t^s,t^r)$ of $X$, using \eqref{lelongtal} 
we get 
$$
\mult_0 X=\int_X M^{(z_1,z_2)}_1 =\int_{\C_t}M^{(t^s,t^r)}_1=\int_{\C_t}M^{t^{\min(s,r)}}_1
=\min(s,r).
$$
This multiplicity is of course well-known, and can be computed in various other ways.
\end{ex}

We will now proceed with some computations of local intersection
numbers. 

\begin{ex}\label{ex1}
Let $X=\{x_2x_1^m-x_3^2=0\}\subset\C^3_{x}$, where $m\geq 1$, and let $A=\{x_2=x_3=0\}$. 
Since $A$ is smooth and contained in $X$, and $X$ is smooth outside the origin in $\C^3$, we must have 
that $A\diamond X(x)=(\mult_x A,0,\ldots,0)$ for $x\neq 0$, cf., \eqref{likhet}.

We shall now  compute the local intersection numbers at $0$.
To this end we consider a generic Vogel sequence of $\J_A$ on $X$ at the origin 
and compute the corresponding Vogel cycle. 
Let $H_1$ be a generic hyperplane that contains $A$, defined by $h_1= \alpha x_2-x_3$. 
Then $X_1=H_1\cdot X$ is the curve
$\{x_2x_1^m-(\alpha x_2)^2=0,\alpha x_2- x_3 =0\}$. 
It follows that $X_1^A$ is equal to $A$, whereas 
$X_1^{Z\setminus A}$ is the curve $\{x_1^m-\alpha^2x_2=0, \alpha x_2-x_3=0\}$. 
Next, let $h_2=\beta x_2-x_3$. Then $X_2=H_2\cdot X_1^{X\setminus A}$ is the cycle 
$\{ x_3=x_2=0,\ x_1^m=0\}$. Since its support is contained in $A$, it is equal
to $X_{2}^A$ and it has order $m$ at the origin. 
We conclude that $V^h=A+m[0]$. Thus
$\epsilon_k(A,X,0)$  is equal to $1$ when $k=\dim A=1$ and $m$ when
$k=0$.

As an illustration, let us also compute $\epsilon_0(A, X,0)=e_2(\J_A,X,0)$ as the Lelong
number of a certain Bochner-Martinelli current. 
Notice that $\tau: (t_1,t_2)\mapsto(t_1^2,t_2^2, t_1^mt_2)$ 
is a surjective, generically $2-1$, mapping $\C^2_t\to X$. 
If $i\colon X \hookrightarrow \C^3$ is the inclusion map we have by Lemma~\ref{listigt} that 
\begin{multline*}
e_2(\J_A,X,0)=\ell_x(M_2^{(x_2,x_3)}\w[X])=
\int_{\C^3}M^{(x_1,x_2,x_3)}_0\w M_2^{(x_2,x_3)}\w [x_2x_1^m-x_3^2]=\\
\int_X M^{(i^*x_1,i^*x_2,i^*x_3)}_0\w M^{(i^*x_2,i^*x_3)}_2=
\frac{1}{2}\int_{\C^2_{t_1,t_2}}M^{(t_1^2,t_2^2,t_1^mt_2)}_0\w  M_2^{(t_2^2,t_1^mt_2)}.
\end{multline*}
According to Theorem~\ref{meanvaluethm}, $M_2^{(t_2^2,t_1^mt_2)}$ is the mean value of all
$$
[(\beta t_2-t_1^m)t_2]\w[(\alpha t_2-t_1^m)t_2]
$$
for generic choices of $\alpha,\beta\in\C$. For generic $\alpha,\beta$,
using  the new  variables $v_1=t_1, v_2=\alpha t_2-t_1^m$,  we get
$$
[\beta t_2-t_1^m]\w[\alpha t_2-t_1^m]=[\beta'v_2-\alpha' v_1^m]\w[v_2]=
[v_1^m]\w[v_2]=m[0]
$$
for some $\alpha',\beta'\in \C$. 
Since $[(\alpha t_2-t_1^m)t_2]=[t_2]+[\alpha t_2-t_1^m]$,
by \eqref{tundra} and \eqref{rakneregel},  we thus have that
$$
[(\beta t_2-t_1^m)t_2]\w[(\alpha t_2-t_1^m)t_2]=
\big([\beta t_2-t_1^m]+[t_2]\big)\w [\alpha t_2-t_1^m]=2m[0].
$$
Now, $M^{(t_1^2,t_2^2,t_1^mt_2)}_0={\bf 1}_{(0,0)}$, so $\epsilon_0(A,X,0)=m$ as expected. 
\end{ex}

The following example is related to Example \ref{hak1} above. 

\begin{ex}\label{hak2}
The mapping $\gamma\colon\C^3_t\to\C^6_z$ defined by
$$
(t_1,t_2,t_3)\mapsto \gamma(t)=(t_1, t_2, t_3t_1, t_3t_2, t_3^2,t_3^3)
$$
is  proper and injective, so that $X:=\gamma(\C^3)$ is a subvariety of $\C^6$. 
Let $A=\{z_3=z_4=z_5=z_6=0\}$. Then $A$ is smooth and contained in $X$ and, since
$X$ is smooth outside $0$, it follows from \eqref{likhet} that
$A\diamond X(x)=(\mult_x A,0,\ldots,0)$
for $x\neq 0$. We want to determine the local intersection numbers
$\epsilon_{3-k}(A,X,0)=e_k(\J_A,X,0)$ at $0$.
Since  $\J_A$ has codimension $1$ in $X$,  
$e_0(0)=0$. Moreover, by King's formula, $M_1^{z_3,z_4,z_5,z_6}$ is a Lelong current on $X$,
and at a point  $x\neq 0$ we know that the Lelong number is $1$ if $x\in A$ and $0$
otherwise. We conclude that  $M_1^{z_3,z_4,z_5,z_6}=[A]$, and hence
$e_1(1)=\mult_0 A=1$.
By Lemma~\ref{listigt},
$$
e_k(A,X,0)=\int_{\C^3_t} M^{\gamma^*z}_{3-k}\w M_k^{\gamma^*(z_3,z_4,z_5,z_6)}=
\int_{\C^3_t}M_{3-k}^{(t_1,t_2,t_3^2)}\w M^{t_3(t_1,t_2,t_3)}_k,
$$
where we have used that 
the ideal $\gamma^* z$ is generated by $t_1,t_2,t_3^2$, that the ideal
$\gamma^*\J_A$ is generated by $t_3(t_1,t_2,t_3)$, and that $e_k(A,X,0)$ only depends on the ideals. 
In light of Example~\ref{hak1} thus,
\begin{multline*}
e_2(A,X,0)=\int_{\C^3_t}M_{1}^{(t_1,t_2,t_3^2)}\w [t_3]\w dd^c\log|t|^2=
\int_{\C^3_t}M_{1}^{(t_1,t_2)}\w [t_3]\w dd^c\log|t|^2 \\
=\int_{\C^2_{(t_1,t_2)}}M_{1}^{t_1,t_2}\w dd^c\log|t'|^2=\ell_0(dd^c\log|t'|^2)=1,
\end{multline*}
where $t'=(t_1,t_2)$.
To see the last equality, in view of Theorem \ref{meanvaluethmlelong}, 
one can replace $dd^c\log|t|^2$ by a 
generic hyperplane $[\alpha\cdot t]$. 
In a similar way one concludes that $e_3(A,X,0)=2$.
\end{ex}



\def\listing#1#2#3#4#5#6{ { #1},\   {\it #2} \   #3
{\bf#4}, \ {#5} \  #6}

\end{document}